\documentclass[lettersize,journal]{IEEEtran}
\usepackage{amsmath,amsfonts}
\usepackage{algorithmic}
\usepackage{algorithm}
\usepackage{array}
\usepackage[caption=false,font=normalsize,labelfont=sf,textfont=sf]{subfig}
\usepackage{textcomp}
\usepackage{stfloats}
\usepackage{url}
\usepackage{verbatim}
\usepackage{graphicx}
\usepackage{cite}
\newtheorem{lemma}{Lemma}
\newtheorem{proposition}{Proposition}
\newtheorem{theorem}{Theorem}
\newtheorem{remark}{Remark}

\usepackage{balance}
\usepackage{xspace}

\newcommand{\eg}{\textit{e.g.},\xspace}
\newcommand{\ie}{\textit{i.e.},\xspace} 
\usepackage{hyperref}
\usepackage{academicons}
\usepackage{xcolor}
\usepackage{hyperref}

\usepackage{tikz}
\newcommand\publishedtext{%
	\footnotesize This is the author's accepted version of the article. The final version published by IEEE is G. O. Ferreira, et al, ``A Joint Optimization Approach for Power-Efficient Heterogeneous OFDMA Radio Access Networks''  IEEE JSAC Special Issue on Advanced Optimization Theory and Algorithms for Next Generation Wireless Communication Networks, vol TBD, pp. TBD, doi: TBD.} 
\newcommand\copyrighttext{%
  \footnotesize \textcopyright 2024 IEEE. Personal use of this material is permitted.
  Permission from IEEE must be obtained for all other uses, in any current or future
  media, including reprinting/republishing this material for advertising or promotional
  purposes, creating new collective works, for resale or redistribution to servers or
  lists, or reuse of any copyrighted component of this work in other works.}
\newcommand\copyrightnotice{%
\begin{tikzpicture}[remember picture,overlay]
\node[anchor=north,yshift=0pt] at (current page.north) {\fbox{\parbox{\dimexpr\textwidth-\fboxsep-\fboxrule\relax}{\publishedtext}}};
\node[anchor=south,yshift=10pt] at (current page.south) {\fbox{\parbox{\dimexpr\textwidth-\fboxsep-\fboxrule\relax}{\copyrighttext}}};
\end{tikzpicture}%
}

\begin{document}

\title{A Joint Optimization Approach for Power-Efficient Heterogeneous OFDMA Radio Access Networks}

\author{\IEEEauthorblockN{Gabriel O. Ferreira\href{https://orcid.org/0000-0002-4592-8975}{\textsuperscript{\includegraphics[scale=0.01]{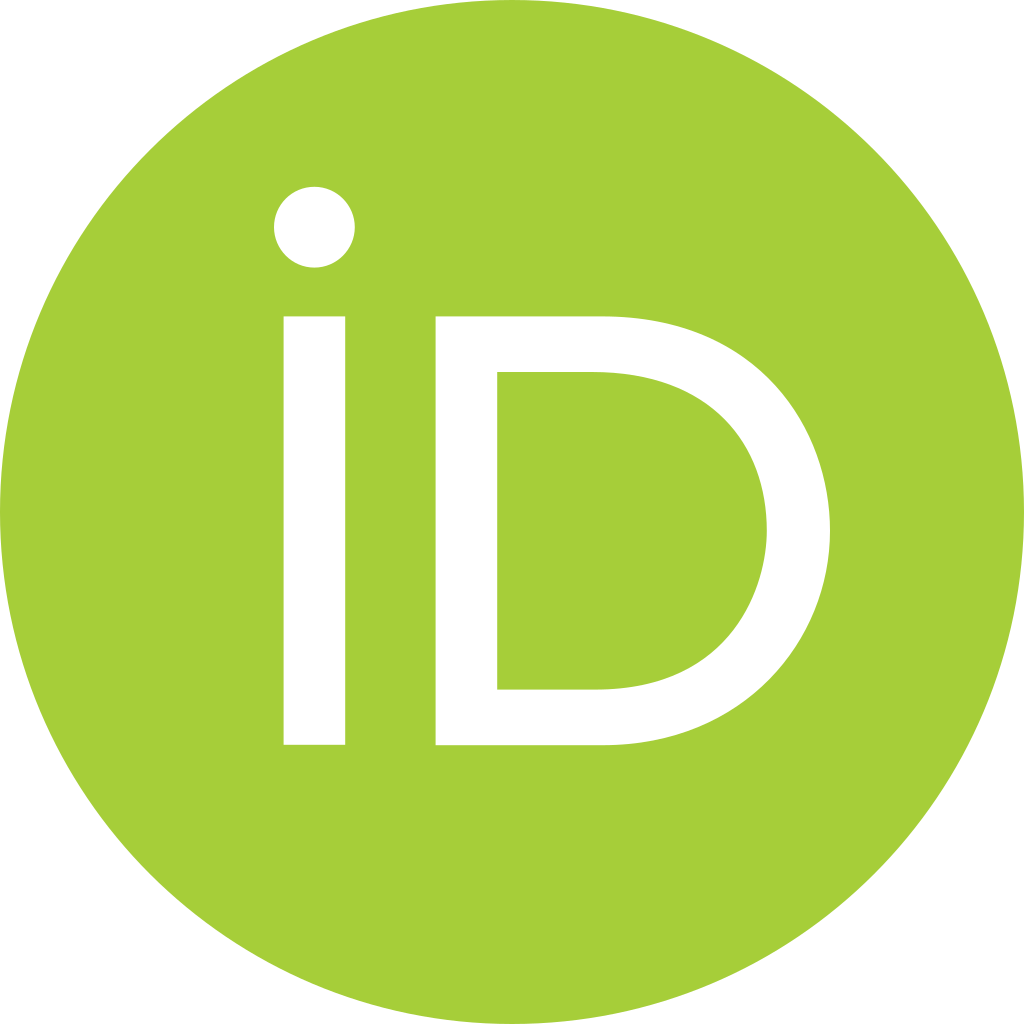}$^{1}$}},
André F. Zanella\href{https://orcid.org/0000-0001-8828-6629}{\textsuperscript{\includegraphics[scale=0.01]{Figures/ORCID_Logo.png}$^{2,3}$}}, 
Stefanos Bakirtzis\href{https://orcid.org/0000-0002-7958-0495}{\textsuperscript{\includegraphics[scale=0.01]{Figures/ORCID_Logo.png}$^{4,6}$}}, 
Chiara Ravazzi\href{https://orcid.org/0000-0002-3186-6195}{\textsuperscript{\includegraphics[scale=0.01]{Figures/ORCID_Logo.png}$^{1}$}}, 
Fabrizio Dabbene\href{https://orcid.org/0000-0002-2258-8971}{\textsuperscript{\includegraphics[scale=0.01]{Figures/ORCID_Logo.png}$^{1}$}},
Giuseppe C. Calafiore\href{https://orcid.org/0000-0002-6428-5653}{\textsuperscript{\includegraphics[scale=0.01]{Figures/ORCID_Logo.png}$^{5}$}}, 
Ian Wassell\href{https://orcid.org/0000-0001-7927-5565}{\textsuperscript{\includegraphics[scale=0.01]{Figures/ORCID_Logo.png}$^{4}$}}, 
Jie Zhang\href{https://orcid.org/0000-0002-3354-0690}{\textsuperscript{\includegraphics[scale=0.01]{Figures/ORCID_Logo.png}$^{6}$}}, 
%
%
Marco Fiore\href{https://orcid.org/0000-0002-0772-9967}{\textsuperscript{\includegraphics[scale=0.01]{Figures/ORCID_Logo.png}$^{2}$}}} 
\thanks{This work was supported by the European Commission through the Horizon 2020 Framework Programme, H2020-MSCA-ITN-2019, MSCAITN-EID, Proposal No. 860239, BANYAN. The work of Stefanos Bakirtzis is supported by the Onassis Foundation and the Foundation for Education and European Culture. \{chiara.ravazzi, fabrizio.dabbene, gabriel.oliveiraferreira\}@ieiit.cnr.it, \{andre.zanella, marco.fiore\}@imdea.org, giuseppe.calafiore@polito.it, \{jie.zhang\}@ranplanwireless.com, 
\{ssb45, ijw24\}@cam.ac.uk}
\thanks{\IEEEauthorblockA{$^{1}$IEIIT CNR Institute, Italy, $^{2}$IMDEA Networks Institute, Spain, $^{3}$Universidad Carlos III de Madrid, Spain, $^{4}$University of Cambridge, UK, $^{5}$Politecnico di Torino, Italy, and CECS, VinUniversity, Vietnam, $^{6}$Ranplan Wireless Network Design Ltd, UK.}
}
}

\maketitle

\copyrightnotice\vspace*{-2pt}

\begin{abstract}
Heterogeneous networks have emerged as a popular solution for accommodating the growing number of connected devices and increasing traffic demands in cellular networks. While offering broader coverage, higher capacity, and lower latency, the escalating energy consumption poses sustainability challenges.
In this paper a novel optimization approach for orthogonal heterogeneous networks is proposed to minimize transmission power while respecting individual users' throughput constraints. The problem is formulated as a mixed integer geometric program, and optimizes at once multiple system variables such as user association, working bandwidth, and base stations transmission powers.
Crucially, the proposed approach becomes a convex optimization problem when user-base station associations are provided.
Evaluations in multiple realistic scenarios from the production mobile network of a major European operator and based on precise channel gains and throughput requirements from measured data validate the effectiveness of the proposed approach.
Overall, our original solution paves the road for greener connectivity by reducing the energy footprint of heterogeneous mobile networks, hence fostering more sustainable communication systems. 
\end{abstract}

\begin{IEEEkeywords}
Heterogeneous mobile networks, orthogonal, energy efficiency, geometric optimization, resource allocation.
\end{IEEEkeywords}

\section{Introduction}
\label{sec:Intro}

\IEEEPARstart{T}{he} efficiency of Radio Access Networks (RAN) is of ever-growing importance to guarantee the combined performance and sustainability of mobile networks. The increasing user traffic demands, the growing number of connected devices and objects, the emergence of new classes of services with very strict Quality of Service (QoS) requirements, and the surging heterogeneity of applications jointly set very high expectations and standards for the operation of 6G RANs~\cite{bc_2021}. As a consequence, RAN optimization is a subject that has received substantial attention from the scientific community, and a wide range of problems have been explored in this area, including how to ensure energy savings via spectral efficiency, transmission power minimization, base stations (BS) deployments, and user association, among others.

Optimizing the many variables that characterize modern RANs is especially challenging in ultra-dense heterogeneous networks (HetNet), where macro, micro, and femto cells are jointly deployed within a relatively small geographical area in order to provide high wireless capacity and enhanced QoS to the local end terminals~\cite{Borralho2021}.
These dense and hierarchical deployments are today a reality, especially in populated urban centers: Figure \ref{fig:hetnet-paris} shows an example of such a RAN configuration observed in a production-grade mobile network, which includes overlapping macro and micro cells providing coverage to the city center of a major European metropolis. In these heterogeneous RAN infrastructures, a lack of optimization of transmission powers, resource allocations, and assignments between users and BSs might lead to high inefficiency, such as high transmission powers and low QoS for users \cite{David2022,Domenico2014}.

In heterogeneous cellular infrastructure scenarios, not only the configurations of the RAN is especially complex, but, if not performed properly, it also exposes the mobile network operator to high expenditures in terms of energy costs that risk to make the whole technology not viable from an economic viewpoint.
The problem is exacerbated in newer generation of Radio Access Technologies (RAT) that are increasingly demanding from an energy perspective: for instance, 5G New Radio is known to be two to three times more energy-consuming than its 4G equivalent~\cite{ericsson2020q4}, due to a need to face much higher traffic volumes per user.
The problem is such that over 90\% of leading mobile network operators have expressed concerns about the rise in energy-induced operating expenses (OPEX)~\cite{AT_2019}.
It is thus unsurprising that many operators and equipment vendors are investing significant effort on developing solutions to make RANs more energy prudent~\cite{Nokia}.

In this paper, we contribute to the endeavour of making dense, heterogeneous RAN deployments greener, by proposing and efficiently solving a novel joint optimization problem that allows minimizing transmission power in orthogonal  networks composed of different types of cells, while meeting individual users' throughput requirements. Unlike many previous works in the field, our optimization approach does not rely on iterative techniques that, in general, can only guarantee local minima through a computationally demanding search, whilst their convergence highly depends on the initial condition selection. In addition, our problem formulation allows considering multiple variables simultaneously, \eg user association, resource allocation, and transmission power, instead of dividing the optimization problem into smaller ones in a \textit{ceteris paribus} manner, as discussed in Section~\ref{sec:RW}.

We achieve the targets above by 
deriving original piecewise concave approximations of the Shannon-Hartley formula, and consequently applying a variable transformation that enables posing the optimization problem into a mixed-integer geometric program (MIGP) form. A differentiating aspect of our study is also that, unlike previous works that rely on simplistic evaluation scenarios, we demonstrate the effectiveness of the proposed optimization framework in dependable settings where up to 800 users generate demands modeled after mobile network traffic measurements and served by real-world heterogeneous RAN deployments, in the presence of wireless channel characteristics reproduced via a high-performance signal propagation solver.

The remainder of the manuscript is organized as follows: in Section \ref{sec:RW} we present the related works, while the target scenario, its constraints, and the problem we aim to solve are analyzed in Section \ref{sec:PF}. The mixed integer geometric program is formulated in Section \ref{sec:MIGP}, and the realistic example is exploited in Section \ref{sec:example}. Finally, conclusions and future works are discussed in Section \ref{sec:C}.

\begin{figure}
    \centering
    \includegraphics[width=1\columnwidth]{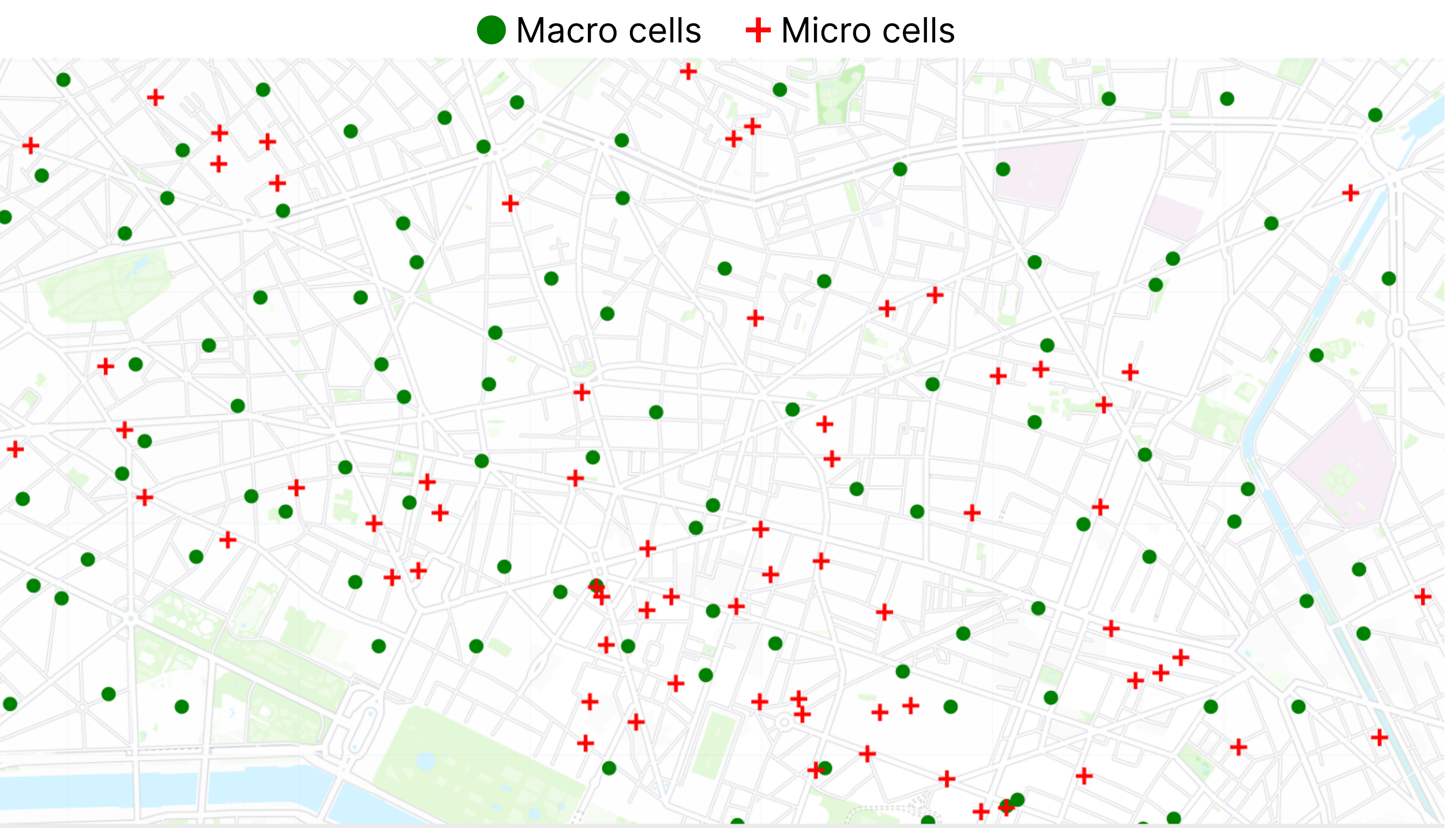}
    \caption{Real-world topology of an operational heterogeneous network deployed in a neighborhood of a large European city.}
    \label{fig:hetnet-paris}
\end{figure}


\section{Related work}
\label{sec:RW}

The approaches proposed in the literature to tackle the HetNet resource allocation problem can be grouped into three general classes: (i) iterative/sequential-based, (ii)  global optimization-based, and (iii) convex optimization-based. With respect to such classification, the solution we propose may be seen as a specific case of the third class, being based on a suitably designed convex approximation of the original problem.
We expound on these approaches separately in the next subsections.  

\subsection{Iterative/sequential approaches}
\label{sub:iterative}

Two types of iterative/sequential approaches can be distinguished. In the first one, the non-convex function is linearized and solved approximately around a point of the optimization space. Then, based on the solution, a new point of the optimization space is selected and this process is repeated until convergence is achieved. The second solves the problem in  \textit{ceteris paribus} manner, splitting the problem variables into two subsets of constant and non-constant parameters. Then,  a  solution to a simplified version of the optimization problem is retrieved  considering only the impact of the non-constant parameters. Consequently, the parameters of the two sets are interchanged, \ie those that were previously constant are now variables and vice versa,   and  a solution for the new problem is computed.  Again, this process is repeated until convergence to a local minimum is achieved. 

Along these lines, in \cite{Foolad2013}, the authors studied the problem of joint user association and resource allocation in heterogeneous networks, assuming that the transmission power at all time is the maximum allowed. A relaxation for an NP-Hard optimization problem was developed, where the new convex formulation is used to obtain upper bounds to the network aggregated throughput. The authors then propose an association rule based on users' signal to interference and noise ratio (SINR) levels and prove its efficacy by comparing the optimal values obtained for the relaxed convex problem with the one when the assignments are obtained with the association rule. Through experiment, the authors show that the optimal values are tight and conclude that their association technique is a simple but yet effective rule.     


Optimization problem for orthogonal frequency division multiple access (OFDMA) heterogeneous networks is studied in \cite{wang2017} for downlink transmissions. The work proposes a joint optimization for user association, sub-channel allocation, and power allocation. The problem is divided into two: first, the powers are fixed and a solution for users association and sub-channel allocation is obtained with the Hungarian algorithm; subsequently, the variables of the first sub-problem are fixed, and a power allocation problem is solved with the difference of two concave functions approximation method. Then, these problems are solved alternately to obtain a local optimal solution. \cite{Nikolaos2016} proposes analytical solutions for association problems subject to backhaul constraints and network topology, allowing each user equipment (UE) at a given location to determine its optimal connecting BS. Subsequently, an iterative algorithm to reach the optimal load of the BS inside the heterogeneous network is derived.  

In \cite{Le2019}, energy efficiency (EE) resource allocation for OFDMA heterogeneous networks with dense deployment is studied. The proposed algorithms have as constraints minimum throughput for delay sensitive users and fairness for delay tolerant users, in addition to the ones imposed by the OFDMA technology. The resource block and power allocations problems, which are non-convex and with combinatorial characteristics, are tackled with three different approaches, named: mixed-integer programming formulation, time-sharing formulation, and sparsity inducing formulation. All of them lead to non-convex problems, which is circumvented using successive convex approximation (SCA) methods, obtaining the problem solution through iterative procedures. 

The iterative/sequential approaches have two main limitations in our context:
 \begin{enumerate}
     \item The optimal solution of sequential approaches are highly dependent on the initial condition provided to the solver, usually leading to local minimum. More importantly, providing a feasible initial resource allocation for large scale problems can be extremely difficult.
     \item The calculation of the gradient (gradient based approach) or
     the two convex functions to represent the non-convex constraints (DC programming) only work for points that are close enough, i.e inside a validity region, from the initial condition. Therefore, small steps inside the decision variables set are taken. These procedures might be time consuming due to possibly many iterations required. Additionally, defining the validity region size is not a trivial task. 
\end{enumerate}

\subsection{Global optimization techniques}

A second class of solution directly tackles the nonconvex optimization problem by means of general global optimization techniques, such as genetic algorithms, particle swarm/ant-colony optimization,
Lyapunov drift plus penalty techniques, and  Graph Theory based stochastic optimization \cite{Liu2020,Bao2014,Bao2015}.
These strategies have been especially studied in the context of heterogeneous cellular networks. For a better overview on 5G resource allocation for heterogeneous networks the reader may refer to the Survey in \cite{xu2021}, where the authors present the taxonomy and strategies applied to such scenarios.

A particular class of solutions we highlight here are those based on neural networks (NNs) approximations. Indeed, NNs, known to be universal approximators of a very large class of functions, have been also applied to resource/power allocation problems. In \cite{Sun2022}, beamforming coordination and sub-carriers/power allocation problems are addressed with the neural network based algorithm transfer learning multi-agent deep Q-network (TL-MADQN). The algorithm is able to extract knowledge from pre-trained agents (BS or channel gains) and adapt to new networks environments, selecting the most appropriate allocation strategy. In \cite{Pan2021}, the problem of joint association and resource allocation to maximize the system throughput is formulated as a mixed integer quadratically constrained quadratic program. The authors decompose the problem into two, and the sub-part associated to the binary variables is solved with a neural network (NN) approach. The proposed NN is based on encoder-decoder architecture and makes use of the well-known long-short term memory (LSTM) and pointer network mechanism.  

Global optimization based approaches generally lead to computationally expensive solutions. More importantly, they usually fail to provide convergence guarantees (usually converging to local optima), and they may require fine tuning of different hyperparameters, which may be difficult in practice. Also, they usually scale badly with the problem dimensions.  
Moreover, in the case of NN-based solutions, they require substantial quantity of data during the training phase, which might not be available in real-world contexts.

\subsection{Convex optimization problems}

Recently, a third line of attack to the problem has emerged, inspired by the work of Boyd \cite{Boyd:03}. These convex optimization-based approaches aim at transforming the original problem by rewriting/approximating it in a convex form. For instance, \cite{Boyd:03} formulates a convex optimization problem based on geometric programming (GP) to solve the problem of simultaneous routing and power allocation in code-division multiple access (CDMA) wireless data networks. GP is also used in \cite{Boyd2002}, where the problems of minimizing transmitter power subject to outage probability and minimizing outage probability subject to transmission power constraints are written in such a format. In that work, an outage happens when the quality of service of a user, measured in terms of signal to interference ratio (SIR), is less than a given threshold value. Many of the works discussed in the previous topic also apply convex optimization techniques, however solving the problem multiple times through successive approximation methods. 

In the authors' opinion, this latter class of approaches represents the most promising solution to our problem for several reasons: i)  convex reformulations lead to problems which are easily solvable with available convex optimization tools, ii) they are guaranteed to converge to the global optimum of the approximation/reformulation, and iii) more importantly, they scale polynomially with the problem size.

On the other hand, this approach has not been further developed for some technical reasons: it is in general difficult to extend the original formulation of \cite{Boyd:03} to more complex setups, due to the fact that the logarithmic  approximation  $\log_2(1+\gamma)\approx \log_2(\gamma)$ introduced in that paper does not easily allow for generalizations (such as optimizing the amount of bandwidth allocated to each user), and becomes loose for small values of $\gamma$.
The work we present here exactly addresses and solves these issues: by proposing a novel approximation based on piecewise power functions, we show that we can still obtain a GP formulation, and, by increasing the number of approximating intervals, tightly approximate the solution.
Moreover, being based on convex programming, our solution does not require a feasible initial resource allocation nor training data availability. 

Finally, we point out that most of the  solutions proposed in the literature aim at maximizing user throughput subject to a fixed bound on the total power consumption. As discussed in the introduction, we choose a different approach here, aiming at minimizing the power consumption under constraints of guaranteed single-user throughput. Indeed, while our approach may be easily adapted to the problem usually considered in the literature, since we formulate both constraints and objective function as convex functions, we feel that the problem formulation discussed here better complies with the increasing request of designing green and sustainable telecommunication networks.

\color{black}

\subsection*{Notation}
\noindent $[n]$ denotes $\{1,\ldots,n\}$. For $x\in\mathbb{R}$, the notation
$\lfloor x \rfloor$ (resp. $\lceil x \rceil$) represents floor($x$) (ceil($x$)), i.e. is the largest (smallest) integer smaller (larger) than $x$.

\section{Problem Formulation}
\label{sec:PF}
Let us consider a scenario where each of $n$ users shall be connected to one base station (BS) among $N$ possible. The downlink transmission rate of a user $i$ connected to a BS $j$ can be calculated with the \textit{Shannon-Hartley Theorem} \cite{Shannon}:
\begin{equation}\label{SINR_0}
    r_{ij} = x_{ij} B_j \log_2(1+S_{ij}(\mathcal{P})),
\end{equation}
\noindent where we define:
\begin{itemize}
    \item[i)] $P_j$: transmission power (W) of one resource block (RB) in BS $j$. Let $\mathcal{P}  \doteq [P_1,\ldots,P_N]$.
    
    \item[ii)] $r_{ij}$: throughput (bits/s) of user $i$ connected to BS $j$.

    \item[iii)] $B_j$: bandwidth (Hz) of BS $j$.
    
    \item[iv)] $x_{ij}$: resources of BS $j$ assigned to user $i$. Hence, $x_{ij}B_{j}$ can be seen as the working bandwidth assigned by BS $j$ to user $i$. We denote by $x\in [0,1]^{n\times N}$ the matrix of the elements $x_{ij}$, with $x_{ij} \in [0,1]$. 

    \item[v)] $S(\mathcal{P})$: Signal-to-Interference-Noise-Ratio. It is the ratio of the power of the wanted signal to the sum of noise and interfering powers of the other BSs, defined as:
    \begin{equation}\label{SINR}
       S_{ij}(\mathcal{P}) = \frac{P_jg_{ij}}{\sigma^{2} + \sum_{k \neq j}P_k g_{ik}}, 
    \end{equation}
 where $\sigma^{2}$ denotes the noise power.
\end{itemize}

A few remarks are at hand regarding the considered setup. 

First, we point out that our formulation considers the case in which the total bandwidth $B_j$ of an antenna is shared equally between a finite number of resource blocks of identical transmission power for the same BS.
This is captured in iv) by the continuous variables $x_{i,j}$, which shall be interpreted as the percentage of resource blocks of BS $j$ assigned to user $i$.
Of course, in real-world applications {$x_{ij}$ are discrete variables. However, if the number of resource block is sufficiently large, this approximation rapidly becomes negligible. Details on how the subcarriers are divided in the considered OFDMA setup, and on how to obtain the discrete variables values given the continuous $x_{ij}$, are provided in Section \ref{sec-subcarrier}.

Second, regarding point v), we note that in equation \eqref{SINR}, $P_j$ is the power of the BS to which the user is connected and $P_k$, $k \neq j$ are the powers of the remaining ones. The quantities $g_{ij}$ and $g_{ik}$ represent the channel gains, which are assumed to be known. Their estimation is discussed in Section \ref{sec:data_gen}.
%
%

\subsection{Constraints}

One of the requirements of the problem is that a user must be connected to just one BS. To specify this constraint, we introduce binary variables $z_{ij}\in\{0,1\},$
$i\in[n]$, $j\in[N]$ as
\begin{equation*}
    z_{ij} = \begin{cases}
      \text{1, if user $i$ is connected to BS $j$,}\\
      \text{0, otherwise.}\\
      \end{cases}
  \end{equation*}
$z\in\{0,1\}^{n\times N}$ is the matrix with elements $z_{i,j}$. Then, the constraints of the considered optimization problem may be written as follows.
\begin{enumerate}[]
\item 
Each user must be connected to just one BS
        \begin{equation}\label{eq:y}
            \sum_{j=1}^{N} z_{ij} = 1, \quad i\in[n].
        \end{equation}

\item
A BS $j$ cannot provide more resources than it has available, that is:
        \begin{equation}\label{eq:x}
            \sum_{i=1}^{n} x_{ij} \leq 1,
            \quad j\in[N].
        \end{equation}
\item 
Each user $i$ has a minimum required throughput $t_i$ (bits/s): 
\begin{equation}\label{eq:r}
    r_{ij} \ge t_i.
\end{equation}
where $t \in \mathbb R_{\ge 0}^{n}$ is the vector with the elements $t_i$.
Combining \eqref{SINR_0} with equation in \eqref{eq:y}, the constraint \eqref{eq:r} can be rewritten in the following equivalent form:
\begin{equation}\label{eq:r2}
\sum_{j=1}^N \left[x_{ij} B_j \log_2(1+S_{ij}(\mathcal{P}))\right] z_{ij} \geq t_i, \quad i\in[n].
\end{equation}
\end{enumerate}

\subsection{Optimization problem}
The optimization problem that minimizes the BS transmission powers and meets the constraints above is written as 
\begin{equation}
    \begin{aligned}
        & \min_{x, z, P} \hspace{0.5cm} \sum_{j=1}^{N} P_j\\
        & \text{s.t.} \,x\in[0,1]^{n\times N}\\
        &\,z\in\{0,1\}^{n\times N}\\
        &\,\mathcal{P}\geq0\\
        & \text{constraints \eqref{eq:y}, \eqref{eq:x} and \eqref{eq:r2}}.
    \end{aligned}
\label{eq:Opt_Problem}
\end{equation}
Note that the above optimization problem is usually very hard to solve, since it entails a combination of binary and continuous variables. More importantly, constraint \eqref{eq:r2} is highly non convex, and also contains a product of continuous and binary variables,
to deal with the requirement of assigning a user to just one BS. To tackle this last issue, we make use of a standard optimization technique known as the \textit{Big-M Method} \cite{Bazaraa,Asghari}.

\begin{lemma}[Big-M trick]\label{lemma:bigM}
Let $P_j \doteq e^{y_j}$ and define 
\begin{equation}\label{eq:s(y)}
f(x_{ij},y)\doteq x_{ij}B_j\log_2(1+S_{ij}(y)),    
\end{equation}
where $y$ is a vector with the elements $y_j$ and with a slight abuse of notation we let $S_{ij}(y)\doteq S_{ij}(P(y))$.
Then, the set of constraints in \eqref{eq:y} and \eqref{eq:r2} can be rewritten as follows 
\begin{align}\label{eq:trick1}
    &f(x_{ij},y)\geq t_i-M \overline{z}_{ij}, \quad i\in[n],  j\in[N],   \\
    \label{eq:trick2}
  & \sum_{j=1}^N \overline{z}_{ij} = N-1, \quad i\in[n]
\end{align}
where $M$ is a sufficiently large constant.
\end{lemma}
\smallskip

\begin{proof}
The constraint in \eqref{eq:trick1} is inactive when $\overline{z}_{ij} = 1$, since ${f}(x_{ij},y) \ge -M$ is always satisfied if $M$ is chosen large enough. On the contrary, $\overline{z}_{ij} = 0$ activates the restriction and the user's throughput requirement is respected. Constraint \eqref{eq:trick2} forces each user to be connected to only one BS (user $i$ is connected to $j$ if and only if $\overline{z}_{ij} = 0$). From \eqref{eq:trick2}, one can write 
\begin{align*}
    &\sum_{j=1}^N 1-{z}_{ij} = N-1, \quad  \sum_{j=1}^N {z}_{ij} = 1, \quad i\in[n],
\end{align*}
repeating constraint in \eqref{eq:y}. By multiplying \eqref{eq:trick1} with $z_{ij}$ and summing over $j$, we have
\begin{align*} 
     &\sum_{j=1}^{N} {f}(x_{ij},y) z_{ij} \geq t_i\sum_{j=1}^{N} z_{ij} -M\sum_{j=1}^{N} {\overline{z}_{ij}} z_{ij}, \quad i\in[n], \\
    &\sum_{j=1}^{N} {f}(x_{ij},y) z_{ij} 
     \geq t_i.
\end{align*}
which is the constraint in \eqref{eq:r2}. Then \ref{eq:trick1}-\ref{eq:trick2} $\Rightarrow$ \eqref{eq:y}-\eqref{eq:r2}. Conversely, \eqref{eq:r2} can be written as
\begin{equation}
    f(x_{ij},y)z_{ij} \geq t_i - \sum_{k\neq j}^{N-1}z_{ik} \hspace{0.1cm}f(x_{ij},y), \hspace{0.2cm} j\in[N].
\end{equation}
Since each user is connected to just one BS, we have that $\sum_{k\neq j}^{N-1}z_{ik} = \bar{z}_{ij}$, leading to
\begin{equation}\label{eq:big_M_1}
     f(x_{ij},y)z_{ij} \geq t_i -\bar{z}_{ij} f(x_{ij},y), \hspace{0.2cm} j\in[N].
\end{equation}
If constraint in \eqref{eq:big_M_1} is satisfied, the following is also respected:
\begin{equation}\label{eq:big_M_2}
     f(x_{ij},y)z_{ij} \geq t_i -\bar{z}_{ij} M, \hspace{0.2cm} j\in[N],
\end{equation}
for $M \gg f(x_{ij},y)$, as
\begin{equation*}
    \begin{cases}
      f(x_{ij},y) \geq t_i, \hspace{0.1cm}\text{if} \hspace{0.1cm} z_{ij}=1,\\
      0 \geq t_i - M,\hspace{0.1cm}\text{if}\hspace{0.1cm}z_{ij}=0.\\
      \end{cases}
  \end{equation*}
Note that the constraint is active when $z_{ij}=1$ and inactive otherwise, exactly like the formulation in \eqref{eq:trick1}-\eqref{eq:trick2}. Then \eqref{eq:y}-\eqref{eq:r2}$\Rightarrow$ \ref{eq:trick1}-\ref{eq:trick2}. Therefore \eqref{eq:y}-\eqref{eq:r2}$\Leftrightarrow$ \ref{eq:trick1}-\ref{eq:trick2}, concluding the proof.
\end{proof}
\section{Mixed Integer Geometric programming}
\label{sec:MIGP}

The proposed approach is based on a piecewise power function approximation of the throughput constraint followed by variables transformation, leading to a mixed integer geometric program. 

\subsection{Piecewise power function approximation}

A key ingredient to our approach is the introduction of a piecewise power function (PPF) approximation. The proposed PPF approach is a method that allows to approximate a function with a set of $m$ simpler concave and monotonically increasing functions in power form. Here the approach is employed to achieve a lower approximation of $\log_2(1+\gamma)$, where $\gamma>0$ represents the signal-to-interference-plus-noise ratio (SINR). By dividing the range of $\gamma$ into intervals, the original function is approximated using a family of $m$ simpler power functions of the form $a\gamma^b$.  This lower approximation ensures that the resulting PPF captures the essential behavior of the original one within each interval.\\
More importantly, we show how this introduced approximation allows to reformulate the original problem into a log-sum-exp form, thus providing a computationally efficient and tractable representation for optimization and mathematical analysis. In a sense, this approach can be seen as a generalization and extension of the original approximation proposed by Boyd in \cite{Boyd:03}, who proposed to approximate $\log_2(1+\gamma)\approx \log_2(\gamma)$.

Formally, we consider to be  given an interval $[0,\Tilde{\gamma}]$. Then, we divide this interval into $m$ subintervals with extremes 
\[
0=\gamma_1,\gamma_2,\ldots, \gamma_m,\gamma_{m+1}=\Tilde\gamma,
\]
and, for $\ell\in[m]$, find $m$ approximations of the form 
\begin{equation}
    \label{eq:phi}
    \varphi_{\ell}(\gamma)=a_{\ell}\gamma^{b_{\ell}},\quad a_{\ell}>0, b_{\ell}\in(0,1)
\end{equation}
valid in the interval $[\gamma_{\ell},\gamma_{\ell+1}]$.


The parameters $a_{\ell},b_{\ell}$ of this function are computed by imposing that the function $\varphi_{\ell}$ in \eqref{eq:phi} interpolates $\log_2(1+\gamma)$ at the extremes $\gamma_{\ell},\gamma_{\ell+1}$. That is, 
we solve for $a_{\ell}$ and $b_{\ell}$ the following system of equations
\begin{equation}\label{eq:gbound}
    \begin{cases}
      \varphi_{\ell}(\gamma_{\ell}) =a_{\ell} \gamma_{l}^{b_{\ell}} = \log_2(1+\gamma_{\ell})\\
      \varphi_{\ell}(\gamma_{\ell+1}) = a_{\ell} \gamma_{\ell+1}^{b_{\ell}} = \log_2(1+\gamma_{\ell+1}).
    \end{cases}
\end{equation}
System \eqref{eq:gbound}
consists of a pair of equations in two unknown variables and, as long as $\gamma_{\ell+1}>\gamma_{\ell}>0$, admits the unique solution
\begin{equation}
\begin{aligned}
\label{eq:ab}
    b_{\ell} &= \frac{\log (\log_2(1+\gamma_{\ell+1}) ) - \log (\log_2(1+\gamma_{\ell}) )}{\log (\gamma_{\ell+1}) - \log(\gamma_{\ell})}, \\ \\
    a_{\ell} &= e^{\log (\log_2(1+\gamma_{\ell+1})) - b_{\ell} \log(\gamma_{\ell+1})}.
\end{aligned}
\end{equation}

The following proposition proves some key properties of the introduced power functions. 

\vskip 3mm
\begin{proposition}
For any 
$\varphi_\ell$ defined as in \eqref{eq:phi}, for $\ell\in[m]$,  we have
\begin{enumerate}
\item $\varphi_\ell(0)=0$;
\item $\varphi_\ell$ is monotonically increasing in $(0,+\infty)$;
\item $\varphi_\ell$ is a concave function in $(0,+\infty)$.
\end{enumerate}
\end{proposition}

\begin{proof}
Point 1 can be directly verified. Regarding points 2-3, the first and second derivatives of $\varphi_\ell(\gamma)$ have the following properties
\begin{align*}
    \varphi_\ell'(\gamma) = \frac{a_\ell b_\ell}{\gamma^{1-b_\ell}} > 0, \hspace{0.1cm} \varphi_\ell''(\gamma) = \frac{a_\ell b_\ell (b_\ell-1)}{\gamma^{2-b_\ell}} < 0. 
\end{align*}
Therefore, the function $\varphi_\ell$ is monotonically increasing and concave in $(0,+\infty)$.
\end{proof}

%
%

\vskip 3mm

\begin{proposition}\label{Prop2}  
The function $\varphi_\ell(\gamma)$ defined as in \eqref{eq:phi}, on the interval $[\gamma_{\ell},\gamma_{\ell+1}]$ with $0<\gamma_{\ell}<\gamma_{\ell+1}<\infty$ has the following property
\begin{equation} \label{eq:interval}
    \begin{cases}
      \varphi_{\ell}(\gamma) \geq \log_2(1+\gamma), & \gamma \in [0,\gamma_{\ell}],\\
      \varphi_{\ell}(\gamma) \leq \log_2(1+\gamma), & \gamma \in [\gamma_{\ell},\gamma_{\ell+1}], \\
      \varphi_{\ell}(\gamma) > \log_2(1+\gamma), & \gamma \in (\gamma_{\ell+1},\infty).
    \end{cases}
  \end{equation}
\end{proposition}

\begin{proof}
It can be easily verified that 0, $\gamma_{\ell}$, and $\gamma_{\ell+1}$ are the only contact points between a given $\varphi_\ell(\gamma)$ and $\log_2(1+\gamma)$.
Finally, notice that since $\varphi_\ell(0)=\log_2(1+0)=0$ and 
\begin{equation}\label{eq:der}
\lim_{\gamma\rightarrow0}\varphi_\ell'(\gamma) = +\infty > \lim_{\gamma\rightarrow0}\log_2(1+\gamma)',
\end{equation}
the first inequality in \eqref{eq:interval} is true in a neighborhood of the origin. 

Moreover, due to the fact that $\varphi_\ell(\gamma)$ and $\log_2(1+\gamma)$ are monotonically increasing and they intercept at  $\gamma_{\ell}$ and $\gamma_{\ell+1}$, second and third inequalities in \eqref{eq:interval} also hold. \end{proof}

\vskip 2mm
The meaning of Proposition \ref{Prop2} is depicted in Figure \ref{fig:exp_app}(a): the power-function approximation $\varphi_\ell(\gamma)$ (red line) is guaranteed to stay below the curve $\log_2(1+\gamma)$ (dashed blue line) inside the interval $[\gamma_{\ell},\gamma_{\ell+1}]$, while it is always above outside the interval. 
We also note that a direct consequence of \eqref{eq:der} is that 
\begin{equation}
    \varphi_{\ell}(\gamma) > \log_2(1+\gamma), \hspace{0.2cm} \text{as} \hspace{0.2cm}\gamma\rightarrow0.
\end{equation}
Therefore, the function $\varphi_{1}(\gamma)$ used to approximate the users' throughput around the origin, where $\gamma_{1}=0$ and $\gamma_{2}>0$, must be linear, i.e, $b_1 = 1$, since we cannot overestimate the throughput capacity of a user
(also note that \eqref{eq:ab} would not have a solution for $\gamma_\ell=0$). The behavior of the linear function $\varphi_1(\gamma)$ is show in Figure \ref{fig:exp_app} (b).

\begin{figure}[tb]
\centering
\subfloat[]{\includegraphics[width=0.8\columnwidth]{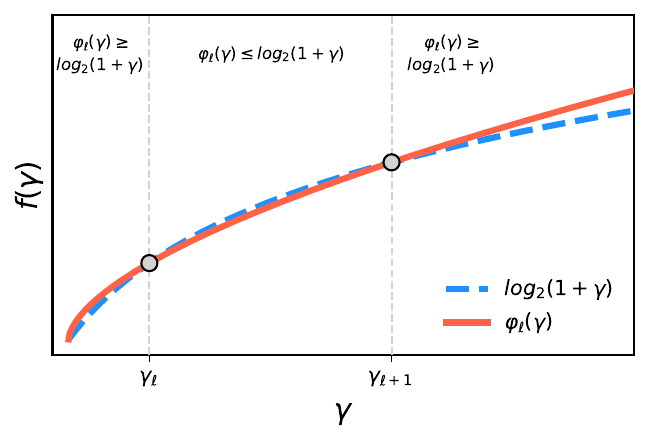}}\\
\subfloat[]{\includegraphics[width=0.8\columnwidth]{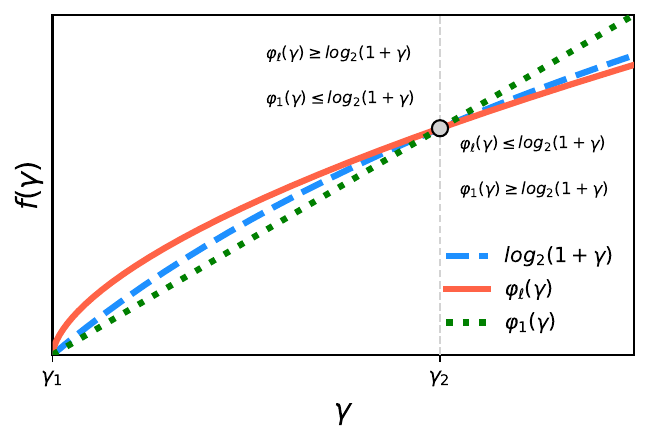}}
\caption{Visual representation of (a) inequalities in \eqref{eq:interval} and (b) requirement of a linear $\varphi_{1}(\gamma)$.} 
\label{fig:exp_app}
\vspace*{-8pt}
\end{figure}

Figure \ref{fig:des} (a)  shows the impact of the hyper-parameter $m$, i.e.\ the number of subintervals of the piecewise approximation. It is clear that $m=2$ leads to a very conservative approximation, since its respective functions strongly underestimate the original one, which could make the problem infeasible. On the contrary, already $m=5$ generates a far better approximation
compared to both $m=2$ and $\log_2(1+\gamma) \simeq \log_2(\gamma)$. 
Clearly, the approximation becomes increasingly better as $m$ increases. The impact of the choice of $m$ is further discussed in Remark \ref{rem1}.
Figure \ref{fig:des} (b) shows the intervals where each of the functions are active. Furthermore, one can see that if the active constraint is respected for a given $\gamma$, the inactive ones are simultaneously respected, since they are always overestimated. As a consequence, solving the optimization problem considering only the active function and its respective interval is equivalent to solving for all $ \varphi_{\ell}(\gamma)$ with $ \ell\in [m]$ simultaneously regardless the value of $\gamma$. Note also that the first function is linear. 
\begin{figure}[tb]
    \centering
    \subfloat[]{\includegraphics[width=0.5\columnwidth, trim={0 0 0 5pt}, clip]{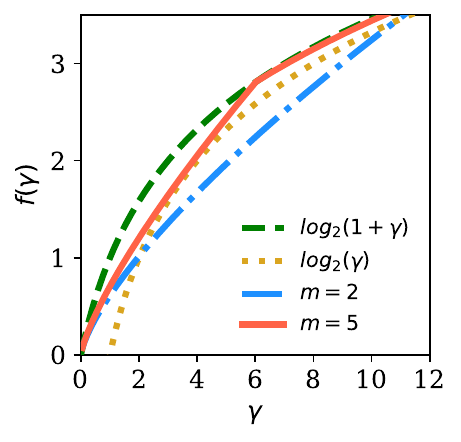}}
    \subfloat[]{\includegraphics[width=0.5\columnwidth, trim={0 0 0 5pt}, clip]{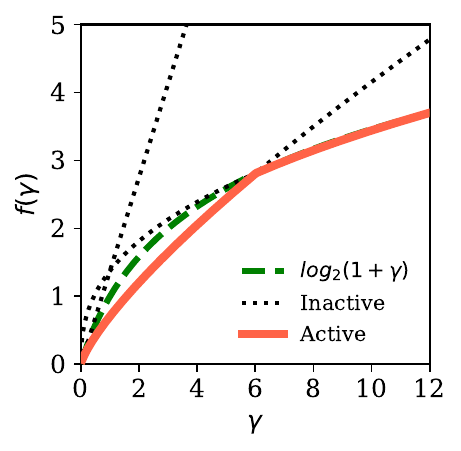}}    
    \caption{(a) Impact of $m$ and (b) constraints active intervals.}
    \label{fig:des}
\end{figure}
\hspace{1cm}\begin{remark}[Interval length $\Tilde{\gamma}$]\label{rem1} The proposed approach is based on PPF approximations inside a given interval $[0, \Tilde{\gamma}]$.
The value $\Tilde{\gamma}$ represents an upper bound on the SINR, and can always be found given bounds on the individual powers $P_j$.
In practice, $\Tilde{\gamma}$ may be considered as a design parameter; to this end, we note that, if $\Tilde{\gamma}$ is very small, the problem may become infeasible since the set of possible solutions would be very conservative. Conversely, a large $\Tilde{\gamma}$ would require a high value of $m$, due to the aforementioned reasons. Despite the fact this might not affect the optimal value (considering a proper approximation), it increases the computational cost since each $\ell\in [m]$ represents a constraint in the optimization problem. For each new function approximation, i.e new $\ell$, $n$ more constraints must be satisfied, since it is related to all users throughput levels, increasing the computation time.
Therefore, $\Tilde{\gamma}$ should be set properly, so the problem is feasible and $m$ is not very large. Note that in case of an infeasible problem (which could happen by not respecting the constraints or the interval $[0, \Tilde{\gamma}]$), both $m$ and $\Tilde{\gamma}$ can be increased. 
In the simulations sections we show how $m$ may be chosen to trade-off optimality and computational complexity.

Finally, we point out that the active interval selection for each $\varphi_\ell(\gamma)$ does not need to follow a specific rule. For instance, one could just divide the approximation interval into $m$ ones of equal length. 
\end{remark}


The importance of the proposed approximation relies in the fact it allows to reformulate our original problem in convex form. This is shown in the following theorem, which represents the main technical result of the paper.

\begin{theorem}
Given $m$ subintervals between $0$ and $\Tilde{\gamma}$, define $m$ approximating functions  as follows
\begin{equation}
    \hat{f_\ell}(q_j,u_{ij}) \doteq \log \left(\frac{\sigma^{2}}{g_{ij}}e^{-q_j-\frac{u_{ij}}{b_\ell}}+\sum_{k \neq j} \frac{g_{ik}}{g_{ij}}e^{q_k-q_j-\frac{u_{ij}}{b_\ell}}\right),
\end{equation}
for $\ell\in[m]$. Then, 
constraint in \eqref{eq:r2} is implied by the following set of  (possibly more conservative) $m$ convex constraints
\begin{equation}\label{eq:fc}
\hat{f_\ell}(q_j,u_{ij}) \leq \frac{\log\left(\frac{B_ja_\ell}{t_i}\right)}{b_\ell} + M \bar z_{ij}, \quad  \ell\in[m], i\in[n], j\in[N].
\end{equation}
Moreover, the conservatism of this approximation vanishes for $m\to\infty$.
\end{theorem}
\begin{proof}
By using the proposed approximation, inequality~\eqref{eq:r2} can be written as
\begin{equation}\label{SINR_app}
    r_{ij} \simeq x_{ij} B_j a_\ell\left(\frac{P_jg_{ij}}{\sigma^{2} + \sum_{k \neq j}P_k g_{ik}}\right)^{b_\ell} \geq t_i, \quad \ell\in [m].
\end{equation}
Then, with some mathematical manipulations and taking the $\log$ of both sides, we have
\begin{equation*}
    a_\ell\left(\frac{\sigma^{2}}{g_{ij}}P_{j}^{-1}x_{ij}^{-\frac{1}{b_\ell}}+\sum_{k \neq j} \frac{g_{ik}}{g_{ij}}P_k P_{j}^{-1}x_{ij}^{-\frac{1}{b_\ell}}\right)^{-b} \geq \frac{t_i}{B_j},
\end{equation*}
\begin{equation*}
    -\log \left(\frac{\sigma^{2}}{g_{ij}}P_{j}^{-1}x_{ij}^{-\frac{1}{b_\ell}}+\sum_{k \neq j} \frac{g_{ik}}{g_{ij}}P_k P_{j}^{-1}x_{ij}^{-\frac{1}{b_\ell}}\right) \geq \frac{\log\left(\frac{t_i}{B_ja_\ell}\right)}{b_\ell}.
\end{equation*}
With the variables transformation $P_j = e^{q_j}$ and $x_{ij} = e^{u_{ij}}$, it is possible to write
\begin{equation*}
    -\log \left(\frac{\sigma^{2}}{g_{ij}}e^{-q_j-\frac{u_{ij}}{b_\ell}}+\sum_{k \neq j} \frac{g_{ik}}{g_{ij}}e^{q_k-q_j-\frac{u_{ij}}{b_\ell}}\right) \geq \frac{\log\left(\frac{t_i}{B_ja_\ell}\right)}{b_\ell},
\end{equation*}
then we have 
\begin{equation}\label{eq:cvc}
\begin{aligned}
    \hat{f_\ell}(q_j,u_{ij}) \leq  \frac{\log\left(\frac{B_ja_\ell}{t_i}\right)}{b_\ell},
\end{aligned}
\end{equation}
which is constraint in equation \eqref{eq:fc} before the Big-M Method. It should be noticed that constraint in \eqref{eq:cvc} is a \textit{log-sum-exp} function, which is convex.
\end{proof}

\subsection{MIGP optimization problem}
The results of the previous section allow us to formulate our main optimization problem. To this end, let $\Tilde{P_{j}}$ be the maximum transmission power of a RB in BS $j$, defined as the ratio between total transmission power of a BS and its number of RBs. Hence, a Mixed-Integer Geometric Programming (MIGP) problem that solves optimization problem \eqref{eq:Opt_Problem} is given by
\begin{equation}\label{min_GP}
\begin{aligned}
 \min_{\bar z_{ij}, u_{ij}, q_j} \quad & \sum_{j=1}^{N} e^{q_j}\\
\text{s.t.: } & e^{q_j}  \leq \Tilde{P_{j}}, \quad j\in[N],\\
  & e^{u_{ij}} \leq 1, \quad  i\in[n], j\in[N]\\
  & \sum_{i=1}^{n} e^{u_{ij}} \leq 1, \quad  j\in[N],\\ 
  & \hat{f_\ell}(q_j,u_{ij}) \leq \frac{\log\left(\frac{B_ja_\ell}{t_i}\right)}{b_\ell} + M \bar z_{ij}, \quad \ell\in[m], \\
  & \hspace{5cm}i\in[n], j\in[N]\\ 
  &  \bar z_{ij} \in \{0,1\}, \quad i\in[n], j\in[N], \\
  & \sum_{j=1}^N \bar z_{ij} = N-1, \quad i\in[n].
\end{aligned}
\end{equation}


Note that the above problem is a mixed-integer one (since it contains the binary variables $\bar z_{ij}$) which solves simultaneously association, assignment, and the power optimization problems. It should be noted that, for given $\bar z_{ij}$, i.e.\ for given association, the problem becomes a classical Geometric Program, i.e.\ a convex optimization problem for which a global optimal solution can be found in polynomial-time, as discussed in the next section. This is an important remark since, as we will see in the application section, there are many practical situations in which simple assignment rules appear to work rather efficiently. In such cases, our approach hence becomes even faster and more scalable. When $\bar z_{ij}$ are problem decision variables, only sub-optimal results are guaranteed. A more detailed discussion on the complexity of the MIGP approach is discussed next.

\subsection{On the complexity of the MIGP algorithm}
\label{sec-complexity}
Geometric Programming, once transformed into its log-sum-exp representation, can be globally solved in polynomial time using standard barrier-based interior-point methods for convex programming. For $n_v$ variables and $n_c$ inequality constraints, $\epsilon$ convergence can be reached in a number of outer (centering) iterations essentially proportional to $\log(n_c/\epsilon)$, whereby each outer iteration requires a number of inner (Newton) iterations that depend on various parameters (Lipschitz constants of gradients and Hessian) but scales very mildly as a double log in the precision, i.e. as $\log\log(c/\epsilon)$, with $c$ a given constant. For the Mixed-Integer GP, a locally optimal solution can be found with branch \& bound methods, which in a worst-case scenario, increases the computational effort exponentially with the problem size. The reader is referred to \cite{Boyd2007} for additional discussions. On the other hand, it should be remarked that the literature presents several (possibly heuristic) methods with low complexity for selecting these integer parameters. Such methods can be surely adapted to the problem at hand.}

\subsection{Sub-carrier assignments given $x$}
\label{sec-subcarrier}

In OFDMA networks, the wide bandwidth transmitted signal is divided into many orthogonal narrow band sub-carriers, such that they do not interfere with each other. Sets of these sub-carriers form a resource block, which can be allocated to the users according to their throughput requirements (represented by the variables $x_{ij} = [0,1]$). Therefore, after solving the aforementioned optimization problem, we still need to assign $\rho_{ij}$ RBs to user $i$ connected to BS $j$ given the respective~$x_{ij}$ computed via optimization problem \eqref{min_GP}. 

Since each RB is allocated exclusively to one user at a given instant, $\rho_{ij}$ should clearly be an integer variable. Ideally, if a user requires a non-integer $\rho_{ij}$, one could just round such a value to its smallest greater integer (ceil function). In practice this is not possible, since the number of RBs is limited and the ceil rounding of all $\rho_{ij}$ would lead to more RBs than the BS can provide. Instead, we propose to solve the optimization problem where the users' throughput requirements and the BSs available working bandwidth are slightly increased and decreased, respectively. As a consequence, some of the users would still respect the throughput requirements even if $\rho_{ij}$ is rounded to its greatest smaller integer (floor function) and some RBs could be allocated to those who do not satisfy such a condition. The procedure to obtain $\rho_{ij}$ given $x_{ij}$ is presented in Algorithm \ref{alg:cap}, with
\begin{itemize}
    \item $\delta_1$: parameter to increase users' throughput requirements;
    \item $\delta_2$: parameter to decrease BS working bandwidth;
    \item $B_0$: bandwidth of one RB.
\end{itemize}

\begin{algorithm}[h!]
\caption{RB assignemtns given $x_{ij}$.}\label{alg:cap}
\begin{algorithmic}
\STATE 
\STATE \textbf{Increase/decrease throughput/working bandwidth:}
\STATE \hspace{0.25cm} $\hat{t}_i \gets (1+\delta_1)t_i$, \quad $i\in[n]$
\STATE \hspace{0.25cm} $\sum_{i=1}^{n} e^{u_{ij}} \leq 1-\delta_2$, \quad $ j\in[N]$
\STATE 
\STATE \textbf{Solve MIGP in \eqref{min_GP}}
\STATE
\STATE \textbf{for} \hspace{0.2cm} {$i\in[n]$}
\STATE \hspace{0.25cm} $j \gets \arg \max (X(i,:))$
\STATE \hspace{0.25cm} $\rho_{ij} \gets \frac{\max (X(i,:))B}{B_0}$
\STATE \hspace{0.25cm} \textbf{if} \hspace{0.2cm} {$\lfloor \rho_{ij} \rfloor B_0 \log_2(1+S_{ij}(y)) \geq t_i$}
\STATE \hspace{0.75cm} $\rho_{ij} \gets \lfloor \rho_{ij} \rfloor$
\hspace{0.5cm} 
\STATE \hspace{0.25cm} \textbf{else}
\STATE \hspace{0.75cm} $\rho_{ij} \gets \lceil \rho_{ij} \rceil$
\STATE \hspace{0.25cm} \textbf{end if}
\STATE
\textbf{end for}
\end{algorithmic}
\label{algorithm2}
\end{algorithm}

Note that one needs to certify that the number of RBs assigned to the users does not exceed the maximum available. In case such a constraint is not respected, $\delta_1$ and $\delta_2$ can be increased. Practically, this means that more users respect the throughput requirement with $\rho_{ij} \gets \lfloor \rho_{ij} \rfloor$ and more RBs are available to users that must use the ceil function. Increasing these parameters can lead to higher transmission power consumption, since they are responsible for increasing throughput demands ($\delta_1$) and decreasing the available bandwidth ($\delta_2$). In extreme cases, the optimization problem might even become infeasible. Consequently, they should be as small as possible.

\section{Application use cases}
\label{sec:example}

We demonstrate the applicability of the proposed algorithm in  few HetNet deployments, based on the real-world configuration of a production-grade RAN. 
To this end, we first outline the process we followed to generate faithful HetNet deployment replicas of an MNO in a major European city. 
Next, we present how MIGP enables energy-efficient resource allocation in a HetNet, benchmarking its performance against other available solutions. 
Finally, we probe the scalability of  MIGP and of the benchmarks, as the complexity of the HetNet, the number of users, and their traffic demands increase.


\subsection{Realistic scenario generation}
\label{sec:data_gen}

We generate our realistic scenarios based on the data obtained from the production network of a leading MNO in a European capital city. The data provided includes the location of the transmitting BS and the configuration of their antennas, \ie the antenna type (macro or micro cell) and steering direction, the technology for that deployment (2G, 3G, 4G, or 5G NSA), and the corresponding statistics about session-level traffic demands. 
This information allows generating a high-fidelity digital twin of the actual network that enables dependable assessments under realistic conditions.

For each scenario, $S$, abiding by the operational network topology we choose a set of base stations, $BS_{j}$, composed of one or more antennas, and we measure the number of sessions arriving on a 1-minute interval of an average busy work day. 
This interval represents a moment where the network operates at higher loads, with the number of connected users and their traffic demands reaching their peak values.
We do not access individual users data in order to comply to privacy rules of the country where the network is located; instead, our data is aggregated at BS level with 1 minute granularity. This means we cannot simply replicate the session-level values, but need to rely on statistics from the measured values. 

For each scenario $S$, we estimate the number of session arriving at each $BS_{j}$ during the time interval of a busy work day. 
From this, we calculate the shares of sessions belonging to each mobile applications, obtaining a heterogeneity of session demands in each scenario, since each mobile service has its own characteristic traffic load. 
For each expected session arriving at $BS_{j}$, we estimate traffic and duration (and thereby throughput) according to the demands of the corresponding mobile service. These demands are calculated based on the statistics collected at BS level across the network over multiple months, which allows us to reliably estimate the session-level traffic probability distribution function of each mobile service, which we use to generate samples of traffic consumed by each session. In this regard, this methodology not only respects the architecture of a major operational network but also the expected density and spatial distribution of users, while complying with user-privacy rules. 

Each generated session is assigned to an individual user placed in a random location, lying within the covered area of $s$-th HetNet topology, determined initially by a Poisson random point process but with additional conditions to respect the boundaries of the street layout of the topology (as to not allocate too many users unrealistically inside buildings). 
We note that while  looking at the traffic and solving the problem for each $BS_{j}$, we consider the full coverage area attributed to all 5G NSA-enabled antennas comprising each $BS_{j}$. 

To evaluate the received signal distribution within the coverage area of each $BS_{j}$ we employ a ray tracing simulator which allows  us to carefully emulate  the electromagnetic wave propagation in the urban fabric \cite{RT_5G_Beyond, Prop_Models_Survey}. That  enables us to quantify the interaction of the waves with objects in the propagation environment and to have accurate estimates of the channel gains $g_{i,j}$ required to evaluate equation (\ref{SINR}). To conduct the site-specific ray tracing simulations for each HetNet topology, we fetch from Open Street Maps information related to the urban environment, such as the city layout and building height, the terrain type, and the foliage.  Then, we place antennas and configure their properties according to the specifications provided by the MNO. A HetNet topology with 1 macro ($BS_{3}$) and 4 micro BSs is shown in Figure \ref{fig:channel_gains}, depicting how the signal emitted by each BS propagates within the urban fabric, and how the micro BSs compensate for the low signal level of the macro cell at larger distances. It is worth mentioning that macro cell $BS_{3}$ is composed of 3 directional antennas, while the remaining micro cells contain a single antenna with an omnidirectional radiation pattern. 

In total, we consider 5 scenarios to evaluate the optimization algorithms, each one assuming different spatial and network complexities. 
Scenario $S_{1}$ is a \textit{downtown} location that will be used in our comparative analysis in Section \ref{sec:benchmark} and \ref{sec:approaches}, containing 1 macro cell and 4 micro cells located in the historical center in the downtown area of the city under consideration. 
The remaining 4 scenarios are utilized in the scalability analysis in Section \ref{sec:scalability}, where we expound on the resilience of MIGP at a variety of different HetNet topologies spread throughout the city. 
Scenarios $S_{2}$ and $S_{3}$ build on the HetNet operating at the same downtown region of $S_{1}$, but we increase the number of micro BS to 7 and 8, respectively, to evaluate the scalability of our solution in denser network configurations. 
Scenario $S_{4}$ represents a less dense HetNet topology located by the \textit{river} side of the city, consisting of  1 macro and 2 micro cells. 
Finally, scenario $S_{5}$ is on an area at the \textit{hill} side of the city, 
composed of 3 macro and 2 micro cells.
All scenarios consider BSs that are 5G NSA-enabled. 

The MIGP depends on the hyper-parameters presented in Table \ref{tab:table_hyper}. The utility of each one was discussed on previous sections. The values of $a_\ell$ and $b_\ell$ allow a piecewise concave approximation of $\log_2 (1+S_{ij}(\mathcal{P}))$ when $S_{ij}(\mathcal{P}) \in [0, \hspace{0.1cm} 513.85]$. Despite the fact there is no closed-form for setting $\delta_1$ and $\delta_2$, we did so by gradually increasing both and checking if the solution to the largest scenario (800 users) could respect the maximum number of RBs (500/BS). Then, the parameters that provided a feasible solution to the largest scenario ($\delta_1$ = 0.05 and $\delta_2 = 0.16$) were also used for all others. $\delta_2 >\delta_1$ because by experimental results, we verified that the $\delta_2$ has more influence on providing a solution inside the physical network limitation on the number of RBs.

\begin{table}[tb]
  \begin{center}
    \caption{Optimization problem hyper-parameters.}
    \label{tab:table_hyper}
    \begin{tabular}{c|c} 
      Hyper-parameter & Value  \\
      \hline \hline
      $\delta_1$ & 0.05 \\
      $\delta_2$ & 0.16 \\
      M &  $10^6$  \\
      $a_\ell$,  $\ell\in [m]$ & [1.408, 0.7330, 1.3150, 1.9061, 3.1232] \\
      $b_\ell$,  $\ell\in [m]$ & [1, 0.7821, 0.4201, 0.2589, 0.1697] \\
      \hline
    \end{tabular}
  \end{center}
\end{table}

\subsection{Qualitative performance analysis}
\label{sec:benchmark}
Initially, we demonstrate the applicability of the proposed framework to a single HetNet deployment of scenario $S_{1}$, benchmarking the performance of MIGP  with other available approaches typically employed for user allocation. To this end, we consider a 5G new radio (NR) access network  operating at 3.5 GHz with 500 resource blocks, utilizing an OFDMA technique with 100 MHz of bandwidth. As previously mentioned, $S_{1}$ consists of 1 macro and 4 micro BSs that accommodate the throughput requirements of 400 users, bounded between 1.35 Kbps and 18.72 Mbps. Pushing for greener communications, via MIGP we aspire to minimize the total transmission power of the HetNet. The noise power is set to -174 dBm/Hz.
\begin{figure}[tb]
\centering
\includegraphics[width = \columnwidth]{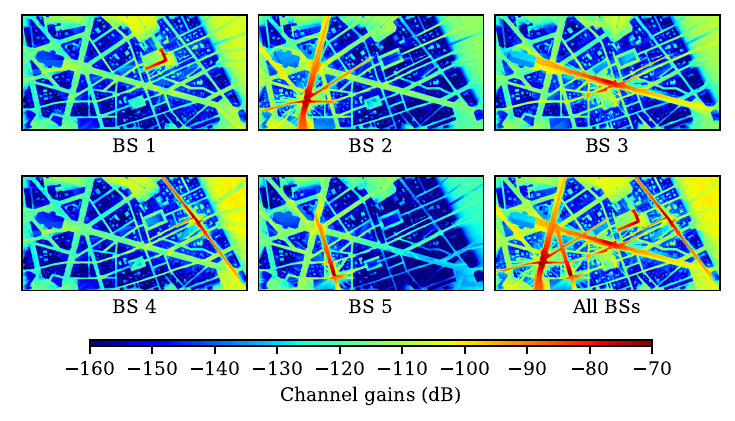}
\vspace*{-8pt}
\caption{Channel gains provided by BSs and the full scenario for $S_{1}$.} 
\label{fig:channel_gains}
\vspace*{-8pt}
\end{figure}
\begin{figure}[tb]
\centering
\includegraphics[width = 0.9\columnwidth]{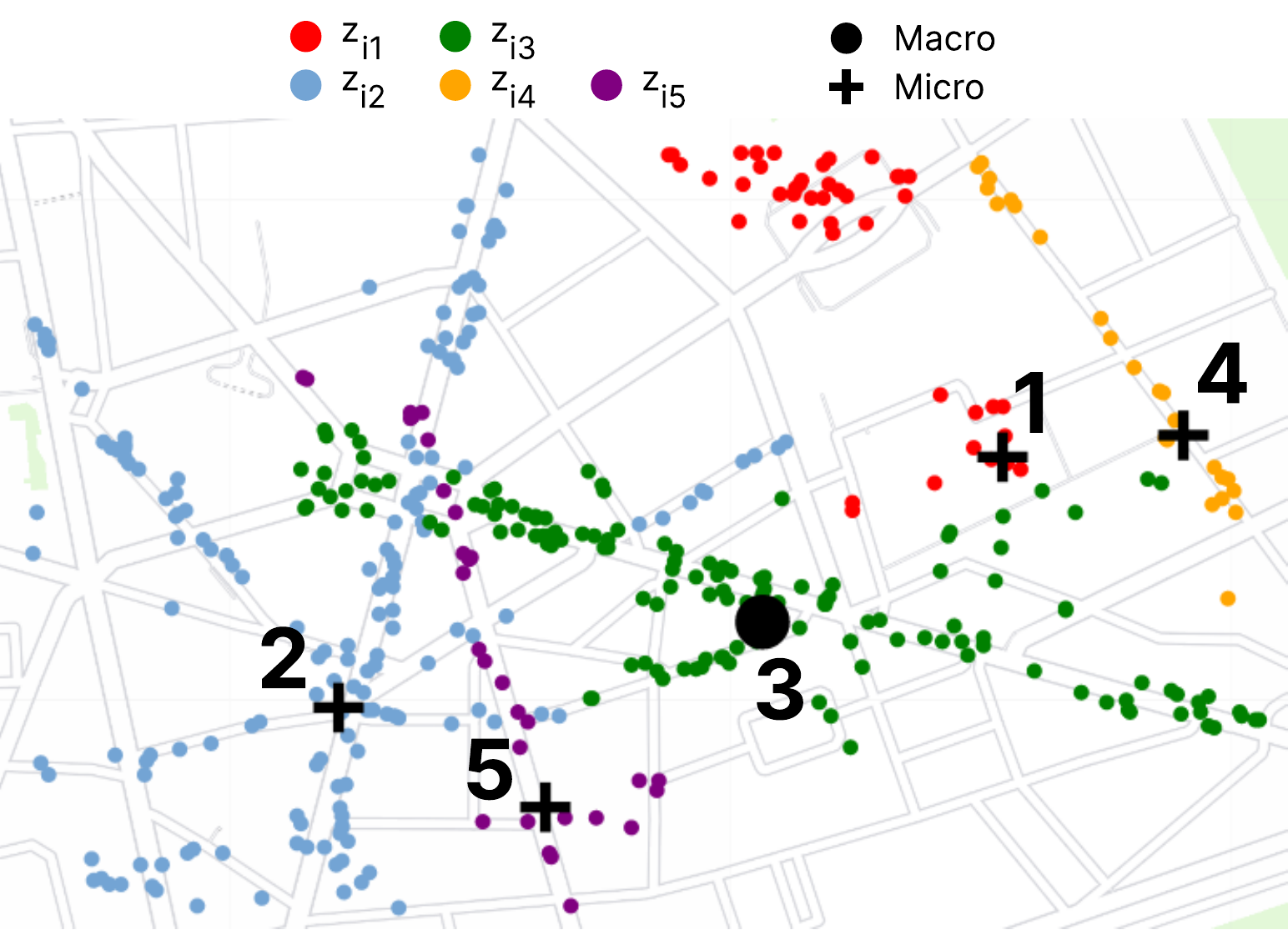}
\caption{
Users' assignment after MIGP optimization for $S_{1}$.
}
\label{fig:assignement}
\end{figure}

The resulting assignment, Figure \ref{fig:assignement}, shows a clear tendency of connecting the users to the BS providing the highest channel gain. We also note on the map a pattern of users being mostly around the outside of buildings, which is amplified by the propagation loss experienced by 3.5MHz communication inside buildings. As previously explained, the working bandwidth variables in $x$ are relaxed to a continuous domain, but in practice we need integer numbers $\rho_{ij}$ representing the number of resource blocks made available to a user. Therefore, after solving the MIGP, such a conversion is done by using Algorithm \ref{alg:cap}. The number of users $n_j$ connected to BS $j$, the number of RB used in each BS, and the respective transmission powers are provided in Table \ref{tab:UA-RB}. $P_j^{\text{Tot}}$ stands for the total transmission power of BS $j$ (each RB consuming $P_j$ Watts). 

\begin{table}[tb]
\begin{center}
\caption{Number of users connected to each BS for $S_{1}$.}
\begin{tabular}{lccc|c}
\multicolumn{1}{c|}{BS $j$} & $n_j$ & $\text{RB}_j$ & $P_j$ (W) & \multicolumn{1}{|c}{$P_j^{\text{Tot}}$ (W)}\\ \hline \hline
\multicolumn{1}{c|}{$j = 1$} & 43 & 423 & $0.2894 \times 10^{-3}$ & 0.1224  \\
\multicolumn{1}{c|}{$j = 2$} &161 & 495 & $0.7117 \times 10^{-3}$ & 0.3523 \\
\multicolumn{1}{c|}{$j = 3$} &142 & 490 & $0.6535 \times 10^{-3}$ & 0.3202 \\
\multicolumn{1}{c|}{$j = 4$} & 24 & 413 & $0.0038 \times 10^{-3}$ & 1.5694 $\times 10^{-3}$\\
\multicolumn{1}{c|}{$j = 5$} & 30 & 435 & $0.2151 \times 10^{-3}$ & 0.0936 \\ \hline
\multicolumn{1}{c|}{$\sum_{j=1}^N $} & 400 & 2256 & 0.0019 &  0.8901   \\ \hline
\label{tab:UA-RB}
\end{tabular}
\end{center}
\vspace*{-24pt}
\end{table}

\subsection{Comparison with global optimization approaches}
\label{sec:approaches}

Since our proposed approach deals simultaneously with the variables $x, z$ and $\mathcal{P}$, we can compare it with some standard procedures performed in realistic network scenarios. As described in \cite{Foolad2013}, the users' association can be done with simple rules, such as \textit{Received Signal Power} or \textit{Maximum channel gain}, defined below.

\begin{itemize}
    \item Received Signal Power (RSP): a user $i$ is associated to the base station $j$ that maximizes the product between transmission power and channel gain, i.e., $j = \arg \max_j(g_{ij}P_j)$.

    \item Maximum channel gain (MCG): a user $i$ is associated to the base station $j$ that maximizes the channel gain between $i$ and $j$, that is, $j = \arg \max_j(g_{ij})$.
\end{itemize}

Note that these simple association rules provide a matrix $z$. Therefore our problem becomes a geometric program (GP), which is convex. In Figure \ref{fig:UE}, we used the MCG rule and refer to this situation as (GP + fixed $z_{ij}$). For $S_{1}$, where each base station $j$ have $n_j$ users assigned via association rule, the resources can be equally shared between the UE according to  
   \begin{equation*}
      x_{ij} = \begin{cases}
      \frac{1}{n_j}, \hspace{0.1cm} \text{if} \hspace{0.2cm} z_{ij}=1, \hspace{0.2cm} \forall i\in[n], \hspace{0.2cm} \forall j\in[N]\\
      \text{0, otherwise.}\\
      \end{cases}
  \end{equation*}
This situation is referred as (GP + fixed $z_{ij}, x_{ij}$). We also compare our approach with its early stage (without Big-M trick, piecewise concave approximation, and variables transformation) given by optimization problem \eqref{eq:Opt_Problem}. Since this is highly non-convex, finding its optimal solution is not trivial and in general just local optimal is achieved. To deal with this problem, we adopt a general global optimization approach. In particular, we compare our MIGP approah with one of the state-of-the-art commercial software for global optimization, 
currently utilized by main players in network optimization
use: the package MIDACO (Mixed Integer Distributed Ant Colony Optimization) solver is able to solve mixed integer non-linear programs with discontinuities, non-convexity, and stochastic noise.  
More details on MIDACO are given in \cite{SCHLUETER2013}.

In Figure \ref{fig:UE}, the optimal value considering the situations described above is calculated for different numbers of users connected to the network during $S_{1}$, keeping the \% of sessions across mobile applications consistent to respect traffic dynamics. The corresponding KWh consumption if the power levels are kept constant for our network snapshot are also shown. The number of users associated to each BS is in Figure \ref{fig:number_of_users}.  

\begin{figure}
    \centering
    \includegraphics[width=1\columnwidth, trim={0 10pt 0 0}, clip]{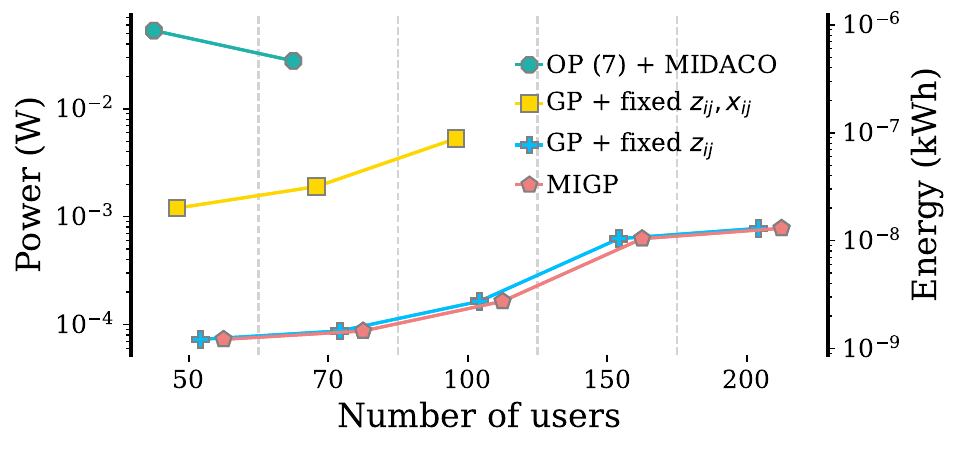}
    \caption{Optimal value as a function of number of users in the network for different approaches for $S_{1}$.}
    \label{fig:UE}
    \vspace*{-4pt}
\end{figure}

\begin{figure}[]
    \centering
    \includegraphics[width = 1\columnwidth, trim={0 5pt 0 0}, clip]{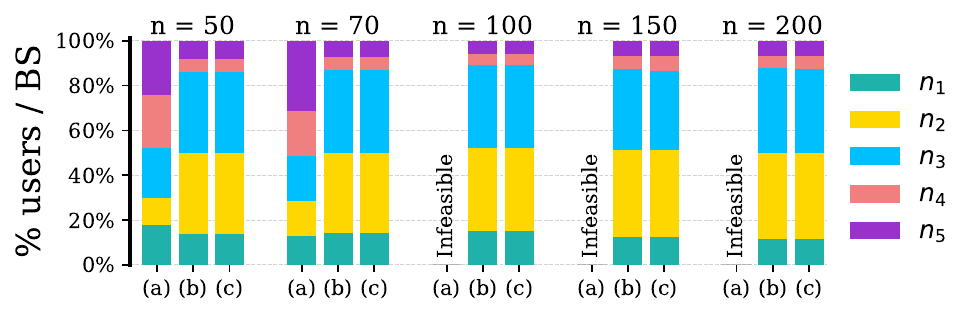}
    \caption{Number of users connected to each BS of $S_{1}$ for (a) OP(7) + MIDACO, (b) GP + fixed $z_{ij}$ and (c) MIGP.}
    \label{fig:number_of_users}
    \vspace*{-4pt}
\end{figure}

From these results, a series of keypoints can be highlighted.
\begin{enumerate}

    \item \textbf{GP+$\text{fixed}$} \boldmath{$z_{ij}$} $\times$ \textbf{MIGP}: 
    the assignments and the corresponding optimal values obtained with MIGP and (GP + \text{fixed} \boldmath{$z_{ij}$}) are almost the same for this network. This happens because the BSs are directional and without large cross-interference, as shown in Figure \ref{fig:channel_gains}. Physically, one could say that the BS deployment is good, which was expected since the data set is from a realistic scenario from a major European city. Regarding the optimization problem point of view, assigning the users according to the MCG rule makes the problem considerably simpler, since it becomes purely convex (can be solved globally and efficiently) and the number of problem variables is reduced, since $z$ would be known in advance. However, there is no guarantee that using the MCG rule in a network where the cross-interference between BSs is large might be a good option. In such a situation, solving the MIGP is more appropriate.

    \item \textbf{OP \eqref{eq:Opt_Problem}+MIDACO}:
    as already discussed, the problem in its most natural format is highly non-convex. We solved it with MIDACO considering that the stop condition is \textit{true} when the optimal value converges and the initial condition is given by the lower bound of each variable. One may notice that the results obtained with both MIGP and (GP + fixed $z_{ij}$) are considerably better in terms of optimal value when compared to (OP \eqref{eq:Opt_Problem} + MIDACO). Therefore, the piecewise concave approximation followed by the variables transformation are indeed efficient mathematical tools to pose the problem in a more suitable format.

    \item \textbf{GP+$\text{fixed}$} \boldmath{$z_{ij},x_{ij}$}:
    fixing the working bandwidth $x$ of users considerably degrades the optimal value, since each one has specifics throughput requirements and channel gains. Furthermore, for $n \geq 150$ the problem becomes infeasible, therefore sharing such a resource equally between users is not a good choice and $x$ should remain as problem variable. Note that the approach of fixing $x$ is rather common in the literature \cite{Parida2014}.

\end{enumerate}

\subsection{Comparison between GP and sequential approaches}
\label{sec:seq}

Parallel to the previous analysis and comparison with global optimization methods, we also performed some numerical comparison between our approach and two ``classical" iterative solutions: i) a successive linearization approach and gradient descent (GD), and ii) a more sophisticated DC-programming approach. In both cases, to perform a fair comparison, we considered the case of fixed assignment, that is case 1) above, with fixed $z_{ij}$. 

For the GD case, a classical linearization of the function 
\[
\log_2\left(1+\frac{P_jg_{ij}}{\sigma^{2} + \sum_{k \neq j}P_k g_{ik}}\right)
\]
around the current feasible point is performed. 
For the second case, we used a difference-of-concave (DC) functions approximation of the form
\begin{equation*}
\begin{split}
& \log_2\left(1+\frac{P_jg_{ij}}{\sigma^{2} + \sum_{k \neq j}P_k g_{ik}}\right) \simeq \\
& \log_2\left(\sigma^{2}+\sum_{j=1}^N P_j g_{ij}\right) - h(\mathcal{P}(0)) - \nabla h(\mathcal{P}(0))^\top (\mathcal{P}\!-\!\mathcal{P}(0)), 
\end{split}
\end{equation*}
with $h(\mathcal{P}) = \log_2\left(\sigma^{2}+\sum_{k\neq j} P_j g_{ij}\right)$ and $\mathcal{P}(0)$ being the point around which each linearization is calculated.

\begin{table}[tb]
\begin{center}
\caption{Comparison between GP and sequential approaches.}
    \begin{tabular}{c|cc|cc} 
     \multicolumn{1}{c|}{} & \multicolumn{2}{c|}{$\mathbf{n = 30}$} & \multicolumn{2}{c}{$\mathbf{n = 40}$} \\
    \cline{2-5}
      Tech. & t(s) & $\sum_j P_j$ & t(s) & $\sum_j P_j$  \\
      \hline \hline
      $\text{GD}_1$ & $217.9$ & $0.2650$ & $266.7$ & $0.2650$\\
      $\text{GD}_2$ & $2230.5$ & $0.2651$ & $2677.5$ & $0.2651$\\
      $\text{DC}_1$ & $471.7$ & $1 \times 10^{-3}$ & $550.5$ & $1 \times 10^{-3}$ \\
      $\text{DC}_2$ & $4701.9$ & $6.52 \times 10^{-4}$ & $5436.8$ & $7.02 \times 10^{-4}$ \\
      \begin{tabular}[x]{@{}c@{}}GP + fixed\\$x_{ij}$, $z_{ij}$\end{tabular} & $6.2$ & $1.68 \times 10^{-4}$ & $8.3$ & $1.98 \times 10^{-4}$\\
      \begin{tabular}[x]{@{}c@{}}GP + fixed\\$z_{ij}$\end{tabular} & $6.1$ & $5.83 \times 10^{-5}$ & $8.2$ & $6.75 \times 10^{-5}$\\
      \hline
      \label{tab:comp}
    \end{tabular}
\end{center}
\end{table}

Due to the limitations of sequential approaches discussed in \ref{sub:iterative}, we set the constraint $||\mathcal{P}-\mathcal{P}(0)||_\infty \leq 10^{-4}$ for $\text{GD}_1$ and $\text{DC}_1$, and $||\mathcal{P}-\mathcal{P}(0)||_\infty \leq 10^{-5}$ for $\text{GD}_2$ and $\text{DC}_2$. The initial condition for the GD, DC, and GP approaches were determined as:
\begin{enumerate}
    \item each user is associated to the BS providing the highest channel gain (all techniques);
    \item the amount of bandwidth allocated to a user is proportional to the respective throughput demand, according to (GDs, DCs, GP + fixed $x_{ij}$, $z_{ij}$):
    \begin{equation*}
        x_{ij} = \frac{t_i}{\sum_{i=1}^n t_i}, \hspace{0.2cm} j\in[N]
    \end{equation*}
    \item $P_j = 0.08$ for all $j\in[N]$ (experimentally verified as a feasible initial condition- GDs and DCs).
\end{enumerate}

The obtained results are presented in Table \ref{tab:comp}. Note that our approaches (GP + fixed $z_{ij}$ and GP + fixed $x_{ij}, z_{ij}$) provide not only dramatically better optimal values, but also requires a substantially smaller runtime. Additionally, the scenarios have just 30 and 40 users because of the complexity in finding initial feasible resource allocation for larger cases, which is necessary for GDs and DCs.

\subsection{Scalability of MIGP across different network topologies}
\label{sec:scalability}

Finally, we test how MIGP performs across different network topologies, which offer variations on the number of BS, density of the deployment, number of users and propagation paths due to the geography of the city. We exemplify the scenarios through Figure \ref{fig:scalability_scenarios}, which shows downtown scenario $S_{3}$ (over same area as $S_{1}$ and $S_{2}$, but with more BS), river scenario $S_{4}$ and hill scenario $S_{5}$. 

Figure \ref{fig:BS} compares the optimal values obtained with MIGP and OP \eqref{eq:Opt_Problem}+MIDACO for the 4 network scenarios, i.e., distinct throughput requirements, channel gains, number of BS, and users. In all cases, the MIGP leads to better results.

\begin{figure}
    \centering
    \includegraphics[width=\columnwidth]{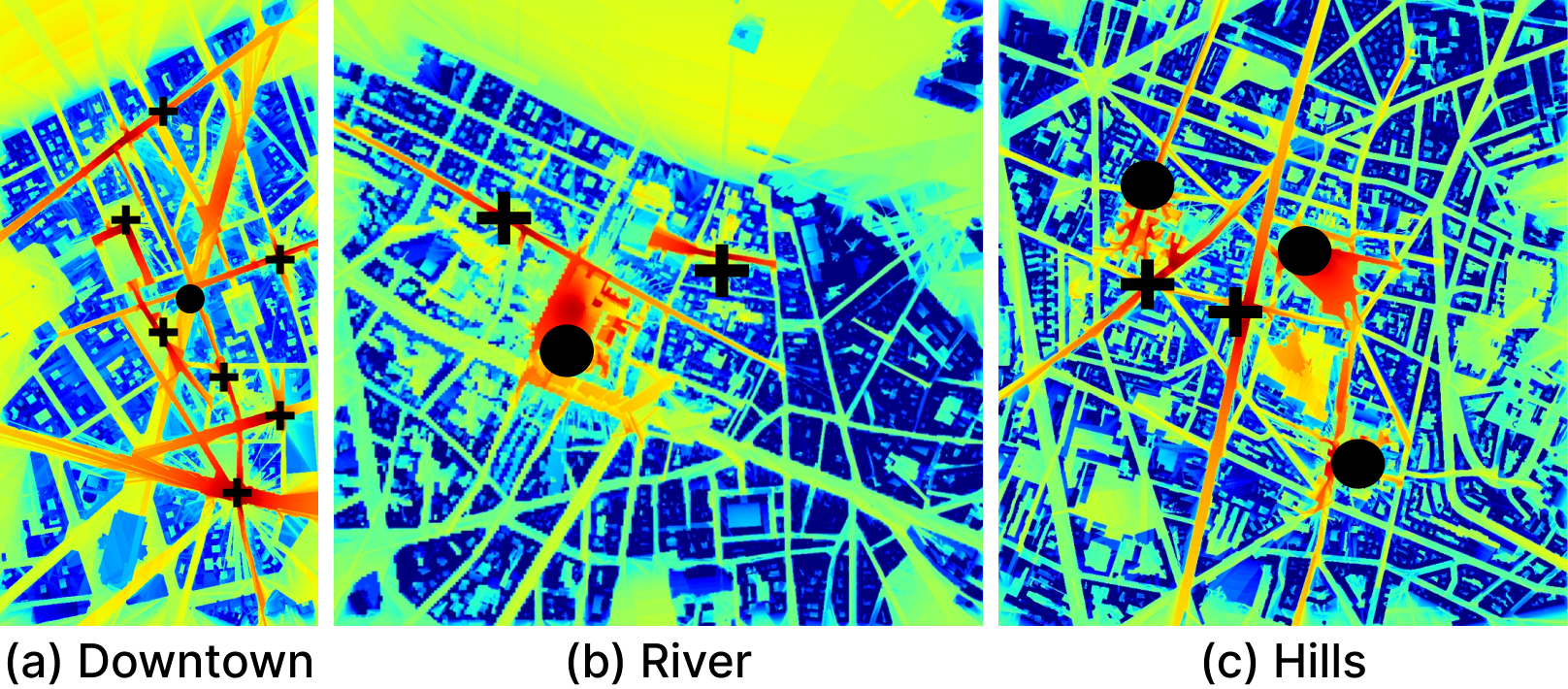}
    \vspace*{-16pt}
    \caption{Channel gains and antenna placements for (a) dense downtown deployment of scenario $S_{3}$, (b) the less dense river deployment of $S_{4}$ and (c) the hill deployment of $S_{5}$.}
    \label{fig:scalability_scenarios}
    \vspace*{-8pt}
\end{figure}

\begin{figure}
    \centering
    \includegraphics[width=0.9\columnwidth]{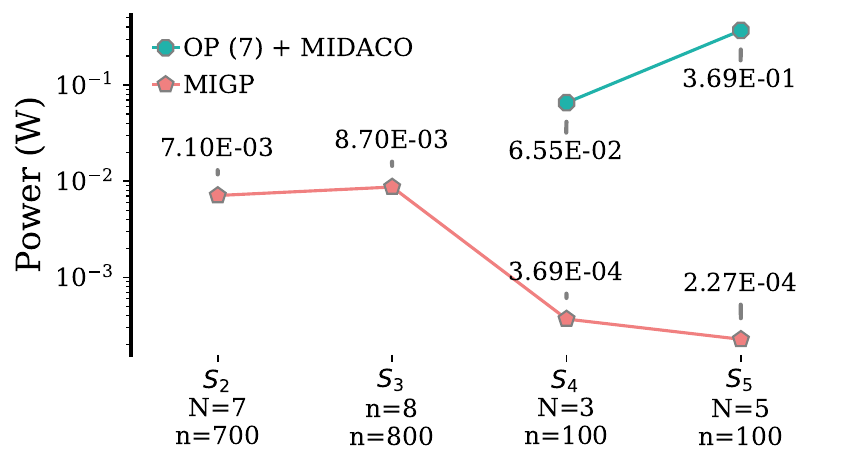}
    \vspace*{-8pt}
    \caption{Power consumption $\sum_{j=1}^N P_j$ for different networks.}
    \label{fig:BS}
    \vspace*{-8pt}
\end{figure}

The results above show that MIGP, contrary to OP \eqref{eq:Opt_Problem}+MIDACO, is able to provide feasible solutions even for scenarios with considerable numbers of BSs and users. Since the association variables are naturally binary, increasing the number of BSs is computationally more expensive than increasing the number of users. However, as previously explained, in environments with good BS deployment (not large cross interference), the MIGP generally associates a user with the BS providing the respective highest channel gain. When using simple association rules, the problem becomes purely convex and can be solved globally and efficiently.   

\subsection{The impact of $m$ on the optimal value and optimization runtime}
\label{sec:m_analysis}

As previously discussed in Remark \ref{rem1}, the number $m$ of concave power functions used in the PPF approximation has an impact both on the goodness of the approximation, and consequently on the obtained optimal value, and on the optimization problem complexity, and thus on the  solution runtime. Therefore, the choice of $m$ should obey a trade-off between accuracy and computational complexity. In particular, we run experiments varying the number of functions $m$, considering network scenarios with 50, 100, and 200 users. As presented in Figure \ref{fig:m_comp}, small values of $m$ lead to solutions with lower execution time, however higher values for the objective function. Conversely, for values of $m$ greater than 15, the time to compute a solution increases without yielding significant improvement on the BSs transmission powers.

Analyzing the optimal values plots of Figure \ref{fig:m_comp}, we observe that the value $m=5$ may be considered a good trade-off, since it represents the point where the curve changes its slope. More generally, we suggested choosing values of $m$  in the interval $5$-$10$.

\begin{figure*}
    \centering
    \includegraphics[width=1\textwidth]{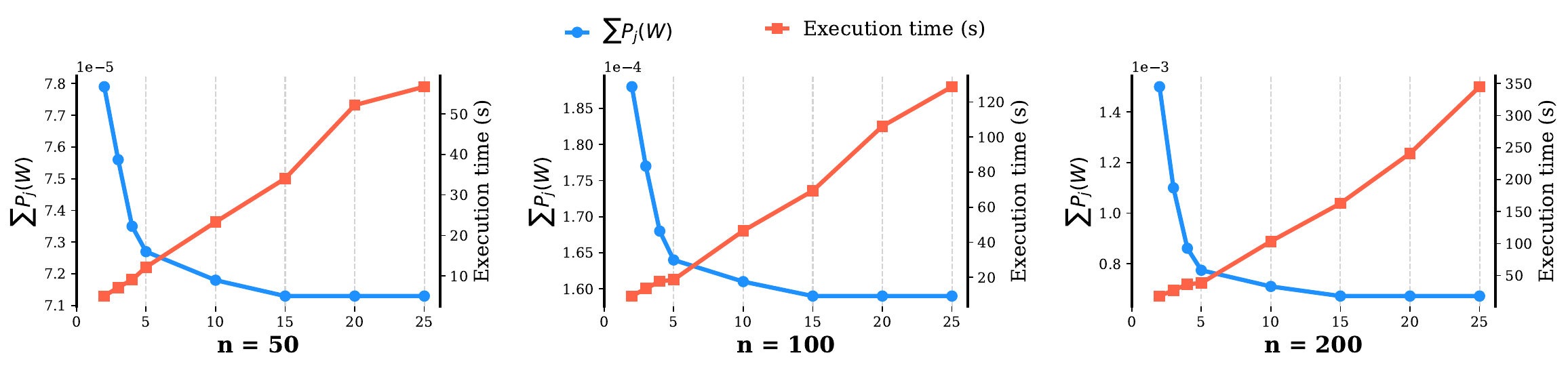}
    \vspace*{-8pt}
    \caption{Optimal cost $\sum_{j=1}^N P_j$ and execution time for different values of $m$.}
    \label{fig:m_comp}
    \vspace*{-8pt}
\end{figure*}

\section{Conclusion}
\label{sec:C}

This research paper introduces an innovative approach to address the challenge of minimizing transmission power in OFDMA heterogeneous networks while ensuring individual users' throughput requirements are met. The proposed method is formulated as a mixed integer geometric program, leveraging the Big-M method to handle the problem's combinatorial nature. Additionally, a piecewise power function approximation of the Shannon-Hartley Theorem is employed, followed by variable transformation to tackle the problem's non-convex characteristics effectively. The optimization problem involves determining the optimal users' working bandwidth, base station (BS) transmission powers, and users' associations. Remarkably, unlike many existing methods, this approach does not rely on successive or iterative algorithms and does not necessitate prior knowledge of a feasible initial condition for the problem. Moreover, it is shown that the problem becomes convex when simple users' association rules are employed.
The effectiveness of the proposed optimization problem is demonstrated through extensive testing in a highly realistic scenario with a considerable number of users. Channel gains are calculated using a high-performance propagation solver, ensuring a practical evaluation.
Looking ahead, we plan to extend our work to include the temporal aspect of the network, accounting for user mobility, varying channel gains, positions, and transmission power over time. In such dynamic scenarios, the challenge of users' handovers will be a critical consideration to address. Additionally, the scenario where the users are served by more than one BS also relies on the SINR calculation, which we converted to the \textit{log-sum-exp} convex formulation. Therefore, we plan to extend our proposed approach to those cases as well. Overall, this paper presents a promising optimization approach for enhancing the energy efficiency of heterogeneous networks while accommodating the increasing demands of connected devices, paving the way for greener and more efficient cellular networks in the future. We plan to share our code, so the MIGP can be easily compared to other approaches in the literature.



 
%

\bibliographystyle{IEEEtran}
\bibliography{JSAC}

\section{Biography Section}

\begin{IEEEbiography}[{\includegraphics[width=1.0in,height=1.25in,clip,keepaspectratio]{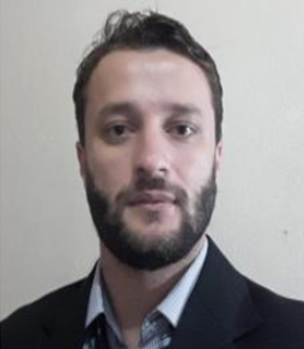}}]{Gabriel O. Ferreira}(Member, IEEE) received the B.S. degree in mechatronics engineering from the Federal Center for Technological Education of Minas Gerais, Brazil, in 2018 and his M.Sc. degree in Electrical Engineering from the Federal University of São João del-Rei, Brazil, in 2020. 

He is currently pursuing the Ph.D. degree with the Department of Electronics and Telecommunication, Politecnico di Torino, Italy. His research interests include modeling and forecasting mobile network traffic, network optimization, and network closed-loop control. Mr Ferreira has received the Marie Skłodowska-Curie Actions-Innovative Training Networks Fellowship.
\end{IEEEbiography}

\begin{IEEEbiography}[{\includegraphics[width=1.0in,height=1.25in,clip,keepaspectratio]{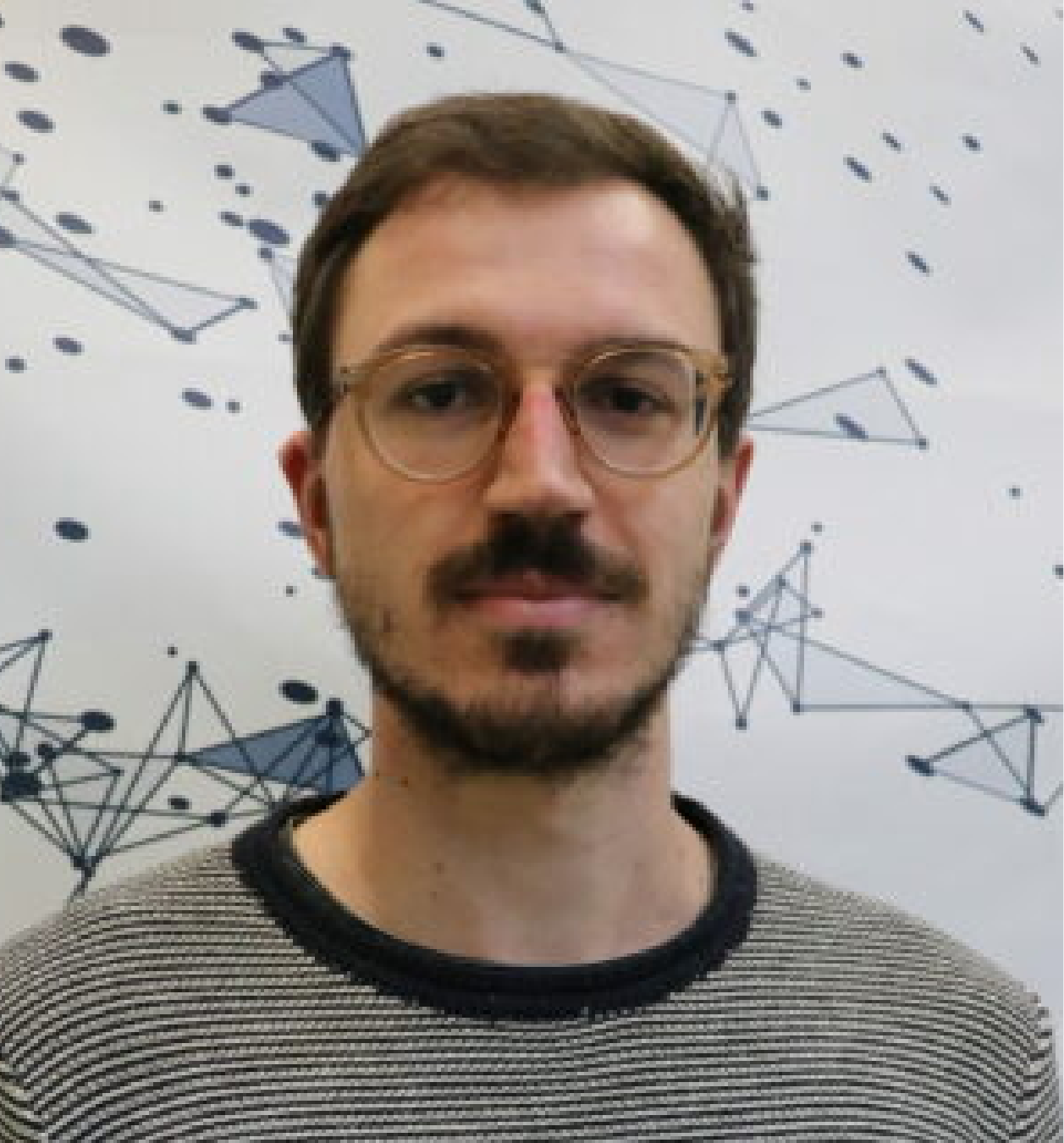}}]{Andre Felipe Zanella} (Member, IEEE) is a PhD researcher at IMDEA Networks Institute and UC3M. He received his bachelor's and master's degrees in Telecommunications Engineering at the Federal University of Parana, in Brazil. He's currently a part of the Networks Data Science group, working on mobile traffic analysis and remote sensing with network metadata, where his interests lay in developing techniques that help solving social sciences problems using information gathered by network operators.
\end{IEEEbiography}

\begin{IEEEbiography}[{\includegraphics[width=1.0in,height=1.25in,clip,keepaspectratio]{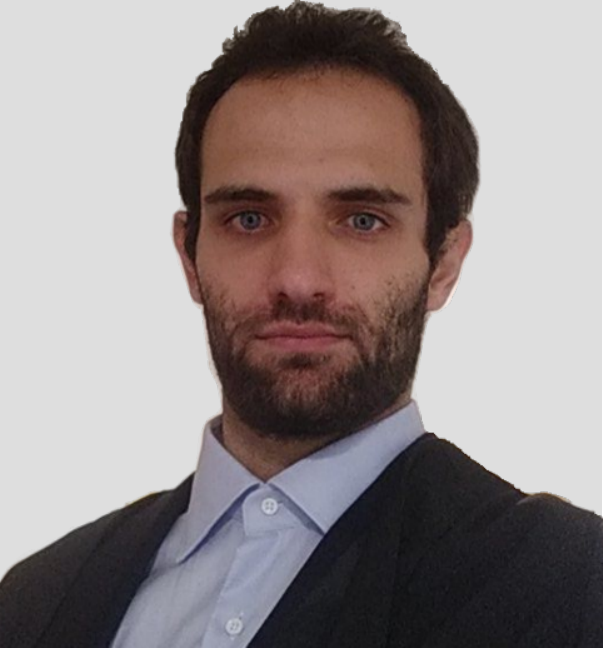}}]{Stefanos Bakirtzis}(Member, IEEE) received the diploma degree in electrical and computer engineering from the National Technical University of Athens, Athens, Greece, in 2017, and the M.A.Sc. degree in electrical and computer engineering from the University of Toronto, Toronto, ON, Canada, in 2020. He is currently pursuing the Ph.D. degree with the Department of Computer Science and Technology, University of Cambridge, Cambridge, U.K.,He is currently working on the Big Data Analytics for Radio Access Networks (BANYAN) project as a member of Ranplan Wireless with the University of Cambridge. His research interests include wireless communication systems, network optimization, wireless channel modeling, machine learning and artificial intelligence, computational modeling, and stochastic uncertainty quantification. Mr. Bakirtzis has received the Onassis Foundation Scholarship, the Foundation for Education and European Culture Grant, and the Marie Skłodowska-Curie Actions-Innovative Training Networks Fellowship.
\end{IEEEbiography}

\begin{IEEEbiography}[{\includegraphics[width=1in,height=1.25in,clip,keepaspectratio]{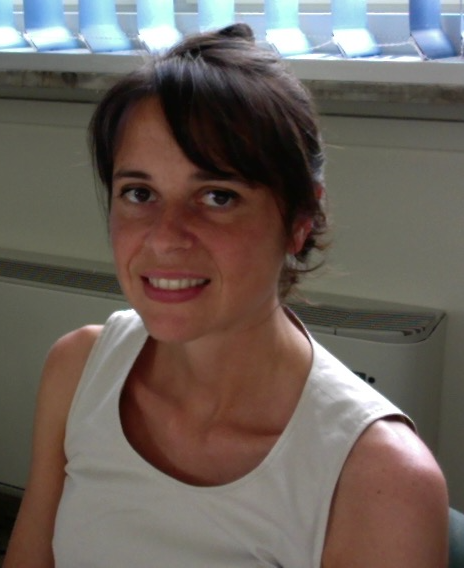}}]{Chiara Ravazzi}(M'13) is a senior researcher at the Italian National Research Council (CNR-IEIIT) and adjunct professor at the Politecnico di Torino. She obtained the Ph.D. in Mathematical Engineering from Politecnico di Torino in 2011. In 2010, she spent a semester as a visiting scholar at the Massachusetts Institute of Technology (LIDS), and from 2011 to 2016, she worked as a post-doctoral researcher at Politecnico di Torino (DISMA, DET). She joined the Institute of Electronics and Information Engineering and Telecommunications (IEIIT) of the National Research Council (CNR) in the role of a tenured researcher (2017-2022). Furthermore, she served as an Associate Editor for IEEE Transactions on Signal Processing from 2019 to 2023, and she currently holds the same position for IEEE Transactions on Control Systems Letters (since 2021) and the European Journal of Control (since 2023). She has achieved the national scientific qualification as a second-tier (associate professor) in the fields of Automatica (09/G1) and Telecommunications (09/F2).
\end{IEEEbiography}

\begin{IEEEbiography}[{\includegraphics[width=1in,height=1.25in,clip,keepaspectratio]{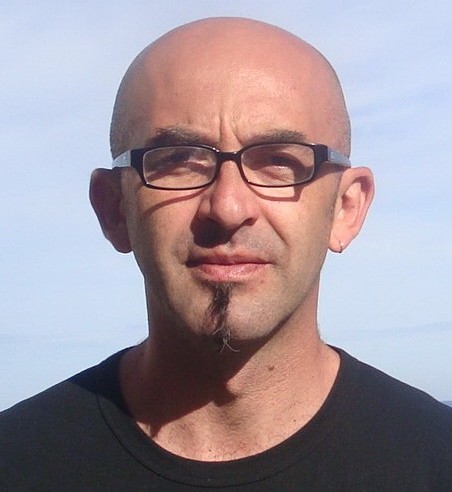}}]{Fabrizio Dabbene} (Senior Member, IEEE) received the Laurea and the Ph.D. degrees from Politecnico di Torino, Italy, in 1995 and 1999, respectively.

He is a Director of Research with the institute IEIIT of the National Research Council of Italy (CNR), Milan, Italy, where he is the coordinator of the Information and Systems Engineering Group. He has held visiting and research positions with The University of Iowa, Iowa City, IA, USA, Penn State University, University Park, PA,
USA, and with the Russian Academy of Sciences, Institute of Control Science, Moscow, Russia. He has authored or coauthored more than 100 research papers and two books.

Dr. Dabbene served as an Associate Editor for Automatica during 2008–2014 and for the IEEE T{\scriptsize RANSACTIONS ON}  A{\scriptsize UTOMATIC} C{\scriptsize ONTROL} during 2008–2012, and he is currently the Senior Editor of the IEEE C{\scriptsize ONTROL} S{\scriptsize YSTEMS} S{\scriptsize OCIETY} L{\scriptsize ETTERS}. He was Elected Member of the Board of Governors during 2014–2016 and he served as the VicePresident for Publications during 2015–2016. He is currently chairing the IEEE-CSS Italy Chapter.
\end{IEEEbiography}

\begin{IEEEbiography}[{\includegraphics[width=1in,height=1.25in,clip,keepaspectratio]{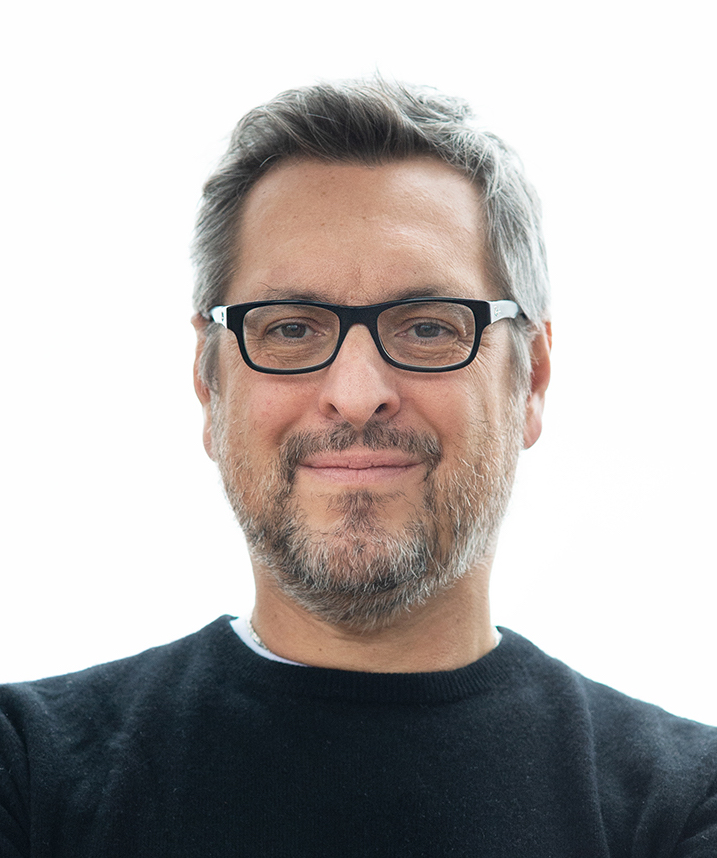}}]{Giuseppe C. Calafiore} (Fellow, IEEE) is a full professor at DET, Politecnico di Torino, where he coordinates the Control Systems and Data Science group. He is also a full professor at the College of Mechanical Engineering at VinUniversity in Hanoi, Vietnam, and an associate fellow of the IEIIT-CNR, Italy. Dr. Calafiore held visiting positions at the Information Systems Laboratory (ISL), Stanford University, California, in 1995; at the Ecole Nationale Sup\'erieure de Techniques Avanc\'ees (ENSTA), Paris, in 1998; and at the University of California at Berkeley, in 1999, 2003, 2007, 2017, 2018 and 2019, and 2021 where he lately taught a Master course on Financial Data Science. He was a Senior Fellow at the Institute of Pure and Applied Mathematics (IPAM), University of California at Los Angeles, in 2010. Dr. Calafiore is the author of about 210 journal and conference proceedings papers, and of eight books. He is a Fellow of the IEEE. He received the IEEE Control System Society “George S. Axelby” Outstanding Paper Award in 2008. His research interests are in the fields of convex optimization, identification and control of uncertain systems, with applications ranging from finance and economic systems to robust control, machine learning, and data science. Dr. Calafiore has over twenty years of teaching experience in Master-level and Ph.D. courses in the areas of Systems and Control Theory, Convex Optimization and Machine Learning.
\end{IEEEbiography}

\begin{IEEEbiography}[{\includegraphics[width=1in,height=1.25in,clip,keepaspectratio]{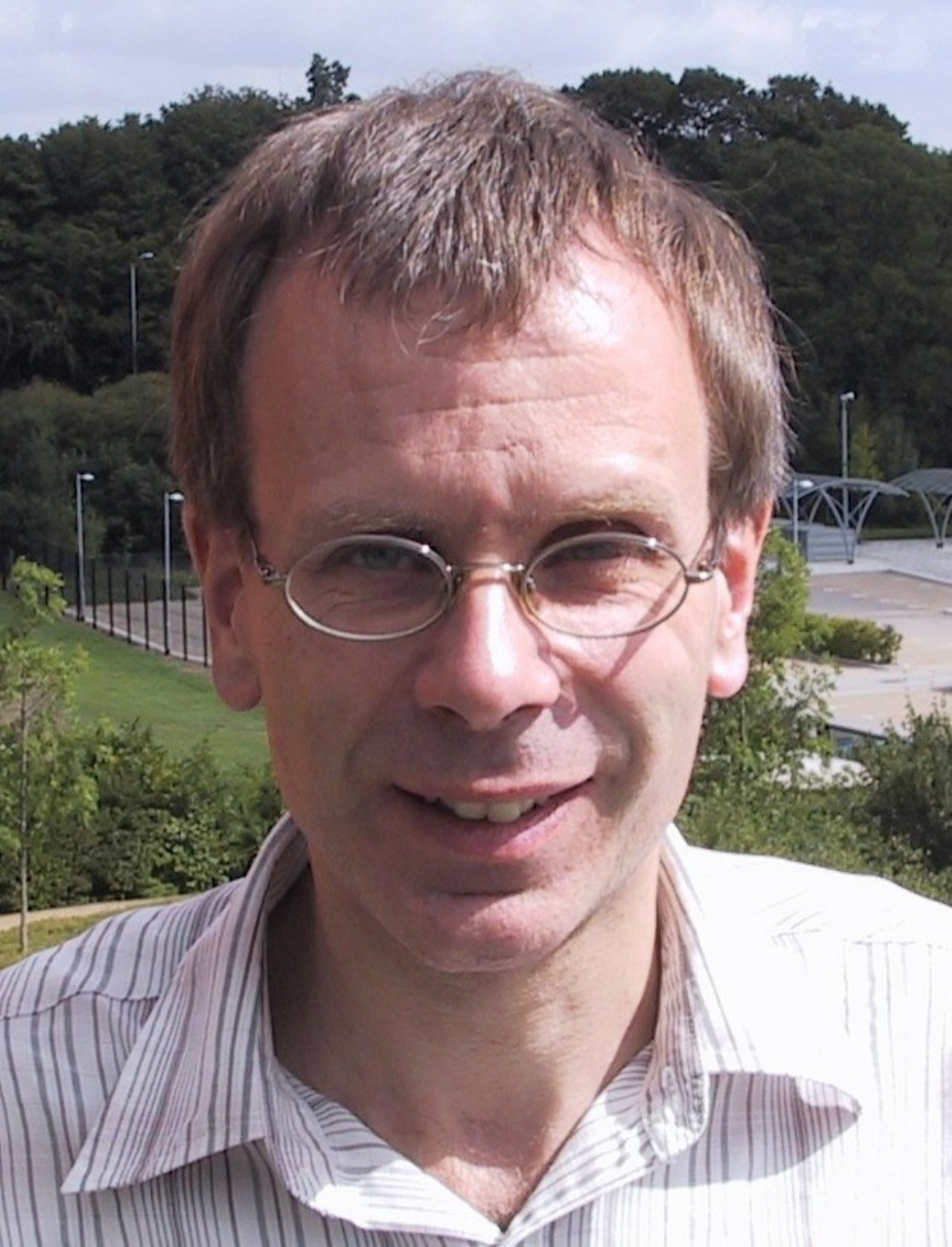}}]{Ian Wassell} (Member, IEEE) received the B.Sc. and B.Eng. degrees from the University of Loughborough, Loughborough, U.K., in 1983, and the Ph.D. degree from
the University of Southampton, Southampton, U.K., in 1990. He is currently a University Associate Professor with the Computer Laboratory, University of Cambridge, Cambridge, U.K., and has experience
in excess of 25 years in simulation and design of radio communication systems gained via a number
of positions in industry and higher education. He has authored more than 200 articles. His current research interests include wireless sensor networks, cooperative wireless networks, propagation modeling, sparse representation, and machine learning. Dr. Wassell is a member of the IET and a Chartered Engineer.
\end{IEEEbiography}

\begin{IEEEbiography}[{\includegraphics[width=1in,height=1.25in,clip,keepaspectratio]{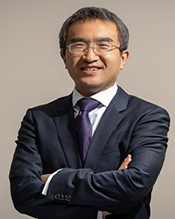}}]{Jie Zhang} (Senior Member, IEEE) has held the Chair in Wireless Systems at the Department of Electronic and Electrical Engineering, University of Sheffield, Sheffield, U.K., on a part-time basis, since January 2011. He is also the Founder, a Board Director and the Chief Scientific Officer of Ranplan Wireless, Cambridge, U.K., a public company listed on Nasdaq First North. Ranplan Wireless produces a suite of world leading indoor and the first joint indoor-outdoor radio access network planning tool -- Ranplan Professional, which is being used by all the top 5 telecom equipment vendors and the world’s largest mobile operators, system integrators and research organizations. Along with his students and colleagues, he has pioneered research in small cell and heterogeneous network and published some of the landmark papers and book on these topics, widely used by both academia and industry. Since 2010, he and his team have also developed ground-breaking work in smart radio environment and building wireless performance modelling, evaluation and optimization, the key concepts of which were introduced in a paper titled “Fundamental Wireless Performance of a Building”, IEEE Wireless Communications, 29(1), 2022. His Google scholar citations are in excess of 8500 with an H-index of 40. He received a Ph.D. degree in industrial automation from East China University of Science and Technology, Shanghai, China, in 1995. Prior to joining the University of Sheffield, he had studied/worked with Imperial College London, Oxford University, and University of Bedfordshire, reaching a status of Senior Lecturer, Reader and Professor in 2002, 2005 and 2006, respectively. 

\end{IEEEbiography}

\begin{IEEEbiography}[{\includegraphics[width=1in,height=1.25in,clip,keepaspectratio]{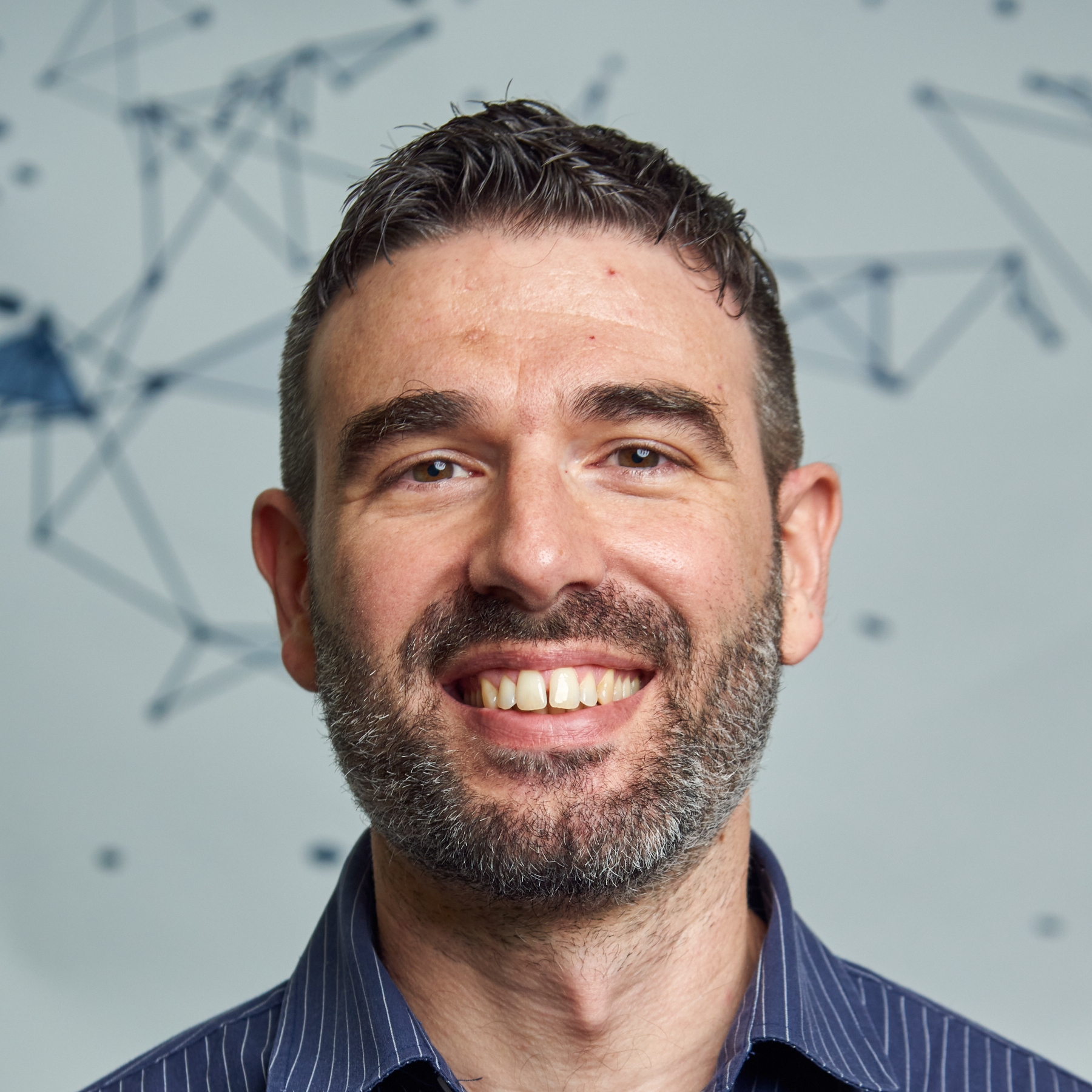}}]{Marco Fiore} (Senior Member, IEEE) is Research Associate Professor at IMDEA Networks Institute and CTO at Net AI Tech Ltd. 

He received M.Sc. degrees from University of Illinois at Chicago, IL, USA (2003), and Politecnico di Torino, Italy (2004), a Ph.D. degree from Politecnico di Torino (2008), Italy, and a Habilitation a Diriger des Recherches (HDR) from Universite de Lyon, France (2014). He held tenured positions as Maitre de Conferences (Associate Professor) at Institut National des Sciences Appliquees (INSA) de Lyon, France (2009-2013), and Researcher at Consiglio Nazionale delle Ricerche (CNR), Italy (2013-2019). He has been a visiting researcher at Rice University, TX, USA (2006-2007), Universitat Politecnica de Catalunya (UPC), Spain (2008), and University College London (UCL), UK (2016-2018). 

Dr. Fiore is a senior member of IEEE, and a member of ACM. He was a recipient of a European Union Marie Curie fellowship and a Royal Society International Exchange fellowship. He leads the Networks Data Science group at IMDEA Networks Institute, which focuses on research at the interface of computer networks, data analysis and machine learning.
\end{IEEEbiography}
\vspace{11pt}

\vfill

\end{document}